\documentclass[11pt]{amsart}
\usepackage{amssymb,mathrsfs}
\usepackage[no-math]{fontspec}
\usepackage[UTF8,scheme=plain,fontset=fandol]{ctex}
\usepackage{indentfirst,latexsym,enumerate}
\usepackage{colortbl}
\usepackage{epsfig,subfigure,graphicx}
\usepackage{amsfonts,amsmath,amscd,amsthm}
\usepackage{enumitem,needspace,placeins,bigints}
\usepackage[all]{xy}
\usepackage[noadjust]{cite}
\usepackage{amsaddr}
\usepackage{mathtools,booktabs,array,longtable}
\usepackage{microtype}
\usepackage[hidelinks,unicode]{hyperref}
\usepackage{bookmark}

\numberwithin{equation}{section}
\allowdisplaybreaks
\theoremstyle{plain}
\newtheorem{theorem}{Theorem}[section]
\newtheorem{lemma}[theorem]{Lemma}
\newtheorem{corollary}[theorem]{Corollary}
\newtheorem{proposition}[theorem]{Proposition}
\theoremstyle{definition}

\theoremstyle{remark}
\newtheorem{remark}[theorem]{Remark}
\newtheorem*{remark*}{Remark}

\newcommand{\R}{\mathbb R}

\newcommand{\U}{\mathcal U}
\newcommand{\cU}{\mathcal U}

\newcommand{\cH}{\mathcal H}
\newcommand{\dd}{\,d}

\newcommand{\sgn}{\operatorname{sgn}}

\setlist{itemsep=2pt,topsep=4pt}
\title[Reflection symmetry in the $1+N$ coorbital problem]{Reflection symmetry in the equal-mass
$1+N$ coorbital problem for $5\leq N\leq8$}
\author{Xiaoxiang Huangfu,  Zijian Meng}
\address{\vspace*{-0.1cm}\upshape
\begin{tabular}{@{}l}
School of Mathematics, Shandong University, Jinan 250100, China.
\end{tabular}}
\keywords{coorbital problem, central configurations, reflection symmetry,
computer-assisted proof}
\date{}

\begin{document}
\begin{abstract}
We study reflection symmetry in the equal-mass Newtonian \(1+N\) coorbital problem for \(N=5,6,7,8\). We prove that every collision-free solution of the angular central-configuration equations has a reflection axis. As a consequence, every non-coalescent limit of planar central configurations consisting of one dominant mass and \(N\) equal small masses is reflection-symmetric about a line through the dominant body.
Our proof combines analytic arguments with rigorous computer-assisted verification. For each possible reflection, we introduce variables that measure the deviation from that reflection and vanish exactly when the configuration has that symmetry. The central-configuration equations yield a symmetric matrix equation in these variables. Concavity of the force derivative and estimates of the resulting quadratic forms reduce the symmetry problem to finitely many rigorously verified inequalities. Although the details depend on \(N\), the same general method applies in all four cases. We also obtain, as a separate result, the complete three-class classification for \(N=5\), up to the natural equivalences.
\end{abstract}
\maketitle
\begingroup
\small
\tableofcontents
\endgroup
\medskip
\section{Introduction}\label{sec:intro}
The planar $1+N$ coorbital problem arises as a limiting central-configuration
problem with one dominant body and $N$ small satellites. More precisely, we
consider masses
\[
1,\epsilon,\ldots,\epsilon
\]
and let $\epsilon\to0$. For a non-coalescent normalized limit, the satellites
lie on a common circle centered at the dominant body, and the
central-configuration equations reduce to a system of angular balance
equations on the circle.

A natural question is whether every solution of these angular equations must
possess a reflection symmetry. Equal satellite masses make the equations
invariant under relabeling and reflection, but this invariance alone does not
imply that an individual solution is symmetric. The purpose of this paper is
to prove that, for $N=5,6,7,8$, every collision-free angular solution does in
fact have a reflection axis.

This question is part of Problem~14 of Albouy, Cabral, and
Santos~\cite[Problem~14]{ACS}. In the non-coalescent planar $1+N$ limit, they
asked whether every configuration must be symmetric when $N\leq8$, and
whether the satellites must form a regular polygon when $N\geq9$. The present
paper concerns the first question for $N=5,6,7,8$.

Our main result is the following.

\begin{theorem}[Reflection symmetry; computer-assisted]\label{thm:main}
Let $N\in\{5,6,7,8\}$, and let
$\theta_0,\ldots,\theta_{N-1}$ be pairwise distinct angular positions modulo one full turn
satisfying the equal-mass Newtonian coorbital angular equations
\eqref{eq:cc}. Then the set of satellite positions is invariant under
reflection about a line through the center of the circle.
\end{theorem}

The theorem concerns the limiting angular problem. It also gives the
following consequence for limits of finite-mass central configurations.

\begin{corollary}[Non-coalescent finite-mass limits]\label{cor:limits}
Consider a convergent normalized sequence of planar central configurations
with masses $1,\epsilon,\ldots,\epsilon$ and $\epsilon\to0$, where the
$N$ satellite limits are distinct and $N\in\{5,6,7,8\}$. Then the limiting
configuration has a reflection axis through the dominant body.
\end{corollary}

The passage from the finite-mass central-configuration equations to the
angular equations \eqref{eq:cc} is established in
Lemma~\ref{lem:limit}. In particular, the possibility that a
satellite approaches the dominant body in the limit is excluded as part of
the argument rather than assumed away.

The implication above is only in the limiting direction. We do not claim
that every solution of the limiting angular problem continues to a family
for positive $\epsilon$, nor do we claim that a finite-mass configuration
with fixed $\epsilon>0$ has the same reflection symmetry. The question for
$N\geq9$ is also outside the scope of the present paper.

The proof of Theorem~\ref{thm:main} is computer-assisted, but it is not based
on a numerical enumeration of stationary points. The analytical part reduces
the symmetry problem to a finite collection of inequalities on compact
parameter regions. These inequalities are then verified rigorously by exact
or interval computations. The continuous implications of the finite checks
are proved in Section~5, while the case-specific verification results are
given in Propositions~\ref{n5:prop:vertices},
\ref{n6:prop:certificate}, \ref{n7:prop:globalcover},
and~\ref{n8:prop:cover}.

\subsection{Relation to earlier work}

The limiting coorbital problem and the classification of its angular
solutions have been studied for many years; see, for example, Cors, Llibre,
and Oll\'e~\cite{CLO} and Albouy--Fu~\cite{AF}. Albouy and Fu obtained the
complete three-class classification for four identical satellites, where
reflection symmetry is an essential step in the classification.

For five satellites, Deng, Hampton, and Wang~\cite[Theorem~6]{DHW} proved
reflection symmetry in the convex case under the corresponding equal-mass
assumptions. Deng and Xu~\cite[Theorem~3.1]{DX} later extended this convex
symmetry result to homogeneous potentials under symmetric mass conditions.
These results are closely related to the present problem, but they do not
cover the full nonconvex equal-mass region considered here.

In the terminology of these works, a convex coorbital configuration is one
for which the center of the circle lies outside the convex hull of the
satellites. Equivalently, the satellites lie in an open semicircle, or some
cyclic gap is larger than $\pi$. In the present Newtonian equal-mass problem,
this large-gap region can be treated directly by the reflection-matrix
argument. The remaining region, in which all gaps are at most $\pi$, requires
the additional analysis developed below.

Together with the earlier results for smaller numbers of satellites reviewed
in~\cite{AF}, Theorem~\ref{thm:main} settles the reflection-symmetry part of
Problem~14 for the cases considered here. We do not attempt to classify all
reflection-symmetric configurations for $N=6,7,8$. For $N=5$, however, the
symmetry result allows a further two-dimensional reduction, and
Section~\ref{sec:n5classification} gives a complete three-class
classification.

\subsection{Outline of the method}

The proof is based on measuring how far a configuration is from a possible
reflection symmetry. For each candidate reflection, we introduce variables measuring this deviation, called reflection defects,
\[
\mathbf t=(t_1,\ldots,t_m)^T
\]
such that
\[
\mathbf t=0
\]
exactly when the configuration is symmetric with respect to that candidate
reflection. The remaining variables, denoted by $\mathbf s$, describe the
corresponding symmetric part of the configuration.

The angular central-configuration equations imply an equation of the form
\[
L(\mathbf s,\mathbf t)\mathbf t=0,
\]
where $L(\mathbf s,\mathbf t)$ is a real symmetric matrix associated with
the chosen reflection. We call it the reflection matrix. Thus, if
$L(\mathbf s,\mathbf t)$ is nonsingular, then necessarily
\[
\mathbf t=0,
\]
and the corresponding reflection symmetry follows.

The main task is therefore to prove nonsingularity of at least one reflection
matrix for every possible solution. A common estimate shows that, in each
case, the reflection matrix is strictly negative on a subspace of codimension
one. Consequently, it is enough to find one direction on which the associated
quadratic form is positive. When such a positive direction is inconvenient
to obtain, we instead prove that the whole matrix is negative definite.
In some remaining regions, it is simpler to rule out a stationary point
directly by showing that one force, or a fixed linear combination of the
forces, cannot vanish.

A concavity property of the force derivative makes these estimates suitable
for rigorous finite verification. In particular, positivity of a fixed
quadratic form on a compact polytope can be reduced to checking its values at
the vertices. This is the main reason why the remaining continuous parameter
domains can be treated by finite exact computations.

The four values of $N$ share this general strategy, but the final reductions
are different. For $N=5$, the remaining problem is reduced by ordering the
gaps. For $N=6$, suitable positive quadratic directions are sufficient. For
$N=7$ and $N=8$, the proof combines matrix estimates with direct force
inequalities; the eight-satellite case also uses an affine matrix upper
bound.

\begin{center}
\begin{tabular}{lll}
\toprule
Satellites & Reflection-matrix size & Final step of the proof\\
\midrule
$5$ & $2$ & Gap-order reduction and finite checks\\
$6$ & $2,3$ & Positive quadratic directions\\
$7$ & $3$ & Matrix and force inequalities\\
$8$ & $3,4$ & Matrix/force inequalities and affine upper bounds\\
\bottomrule
\end{tabular}
\end{center}

Sections~2--5 develop the common equations, analytic estimates, and rigorous
finite-verification principles. Sections~6--9 treat the cases
$N=5,6,7,8$, respectively. Section~\ref{sec:n5classification} gives the
additional classification result for five satellites, and the final section
summarizes the logical dependence of the computer-assisted parts of the
proof.

\section{Equations, normalization, and the limiting problem}\label{sec:equations}
\subsection{The angular problem}
Write the satellite positions in positive cyclic order, fix a rotational
representative by $\theta_0=0$, and set
\begin{equation}\label{eq:gaps}
 x_i=\theta_{i+1}-\theta_i>0,\qquad
 \theta_N=\theta_0+2\pi,\qquad
 \sum_{i=0}^{N-1}x_i=2\pi,\qquad R_i=x_i/\pi.
\end{equation}
Indices on gaps are cyclic. A later reflection chart may choose a different
rotational origin; the force equations are unchanged. Let $\theta_{ij}\in(0,2\pi)$ be the positive
directed separation from satellite $i$ to satellite $j$, with the
angle difference interpreted modulo one full turn. We use
\begin{align}
 f(t)&=\sin t\left((2\sin(t/2))^{-3}-1\right),\quad 0<t<2\pi,
 \label{eq:f}\\
 F_i(\theta)&=\sum_{j\ne i}f(\theta_{ij})=0,
 \quad 0\leq i<N.\label{eq:cc}
\end{align}
In particular $f(2\pi-t)=-f(t)$ and $\sum_iF_i=0$.

Define the even periodic pair potential, away from its collision points, by
\begin{equation}\label{eq:W}
 \widehat W(t)=\cos t-\frac1{2|\sin(t/2)|},\qquad
 W(t)=\cos t-\frac1{2\sin(t/2)}\quad(0<t<2\pi).
\end{equation}
Then $W'=f$, and the full angular potential satisfies
\begin{equation}\label{eq:potential}
 \mathcal U(\theta)=\sum_{0\leq i<j<N}\widehat W(\theta_j-\theta_i),
 \qquad \partial_{\theta_i}\mathcal U=-F_i.
\end{equation}
This effective potential is not the positive Newtonian potential of the
unreduced problem. Its sign must be kept when interpreting Hessians and
force derivatives. Criticality is taken on the collision-free angular
configuration space modulo common rotation. We do not quotient out
reflections when formulating the equations; cyclic relabeling and reversal
are used only to choose covering representatives. When counting shapes
for five satellites, these two operations generate the dihedral action
$D_5$ on the cyclic gaps.

\subsection{Why the angular equations apply to the original problem}
For positive masses $m_i$, use the positive Newtonian potential
$U_{\rm Newt}=\sum_{i<j}m_im_j/|q_i-q_j|$ and the central-configuration
convention
\[
 \nabla_{q_i}U_{\rm Newt}=-\lambda m_i(q_i-c).
\]
In the next lemma only, the primary is labeled separately and the
satellites are numbered $1,\ldots,N$. Translating the primary to zero
is not the same as fixing the center of mass at zero.

\begin{lemma}[Necessary non-coalescent limit equation]\label{lem:limit}
Let $N\geq3$. Consider a convergent sequence of planar central
configurations of masses $1,\epsilon,\ldots,\epsilon$, with
$\epsilon\to0$. Translate the primary to zero and normalize
$\max_i|r_i|=1$, where $r_i$ are the relative satellite positions.
If the satellite limits are distinct, all limiting radii equal one and
the limiting angles satisfy \eqref{eq:cc}.
\end{lemma}
\begin{proof}
Subtract the primary equation from the satellite equation. This removes
the center-of-mass term and gives
\begin{equation}\label{eq:finite}
 -\frac{r_i}{|r_i|^3}
 +\epsilon\sum_{j\ne i}\frac{r_j-r_i}{|r_j-r_i|^3}
 -\epsilon\sum_j\frac{r_j}{|r_j|^3}
 =-\lambda r_i.
\end{equation}
Suppose one satellite tends to the primary. Distinct satellite limits
imply that it is the only one, and all satellite--satellite distances
remain bounded away from zero. Put $\rho=|r_i|\to0$. Isolating its term
in the last sum, and taking the scalar product with $r_i/\rho^2$, yields
\[
 \lambda=\frac{1+\epsilon}{\rho^3}+O(\epsilon/\rho).
\]
For another satellite, whose radius stays bounded below, the same relative
equation instead gives $|\lambda|=O(1+\epsilon/\rho^2)$. Multiplying
both statements by $\rho^3$ gives respectively limits one and zero,
a contradiction. No satellite limit is the primary.

All terms remaining after the explicit factors $\epsilon$ in
\eqref{eq:finite} are now bounded. Fixing one satellite bounds $\lambda$;
every convergent multiplier subsequence has limit satisfying
$\lambda_*r_i^*=r_i^*/|r_i^*|^3$. The nonzero limiting radii are equal,
and their maximum is one. Thus every radius is one and $\lambda_*=1$.

Take the planar cross product of \eqref{eq:finite} with the current
vector $r_i$, and divide by $\epsilon$ before taking the limit. The
radial terms vanish exactly. We obtain
\[
 \sum_{j\ne i}(r_i\wedge r_j)
 \left(|r_j-r_i|^{-3}-|r_j|^{-3}\right)=0.
\]
All denominators have nonzero limits. The unit-radius limit is precisely
\eqref{eq:cc}, including its indirect-force term $-\sin\theta_{ij}$.
\end{proof}

Only this necessary direction is used. An angular solution is not asserted
to extend to a finite-$\epsilon$ family, and symmetry of its limit is
not asserted to imply symmetry at positive $\epsilon$.

\section{Variables measuring deviation from reflection symmetry}\label{sec:common}
\subsection{Concavity of the force derivative}
The estimates below use strict concavity of the force derivative. This
property allows us to control divided differences and quadratic forms. Write $g=f'$ and $s=\sin(t/2)$.
Direct differentiation gives
\begin{equation}\label{eq:g}
 g(t)=-\frac1{4s^3}+\frac1{8s}-1+2s^2,
 \qquad g(\pi)=\frac78,\quad g(\pi/2)=-\frac{3\sqrt2}{8}.
\end{equation}

\begin{lemma}[Concavity of the divided difference]\label{lem:phi}
The function $g$ is strictly concave on $(0,2\pi)$, symmetric about
$\pi$, and strictly increasing on $(0,\pi)$. On the open convex domain
$|h|<\min\{z,2\pi-z\}$, define the divided difference below, using
its continuous value when the two arguments coincide. Its sign is retained
throughout the argument. The function
\begin{equation}\label{eq:phi}
 \Phi(z,h)=\frac12\int_{-1}^1g(z+hv)\dd v
 =\begin{cases}\dfrac{f(z+h)-f(z-h)}{2h},&h\ne0,\\g(z),&h=0\end{cases}
\end{equation}
is jointly strictly concave, even in $h$, and satisfies
\begin{equation}\label{eq:phi-properties}
 \Phi(2\pi-z,h)=\Phi(z,h),\qquad \Phi(z,h)\leq g(z).
\end{equation}
For fixed $h$, it is strictly increasing in $z<\pi$ on its domain;
for fixed $z$, it is strictly decreasing as $h>0$ increases.
\end{lemma}
\begin{proof}
Differentiation of \eqref{eq:g} gives
\[
 g''(t)=-\frac{64s^7-32s^5+s^4-20s^2+24}{32s^5}<0.
\]
If $s^2\geq1/2$, the first two numerator terms have nonnegative sum
and $s^4-20s^2+24\geq5$. If $s^2\leq1/2$, the numerator is at least
$14-4\sqrt2>0$. Reflection about $\pi$ now gives the stated increase.
For $(a,b)\ne(0,0)$,
\[
 D^2\Phi(z,h)[(a,b),(a,b)]
 =\frac12\int_{-1}^1g''(z+hv)(a+bv)^2\dd v<0.
\]
The symmetry identities are immediate and the upper bound is Jensen's
inequality. Symmetry and strict concavity in $z$ imply its monotonicity.
For $h>0$, differentiation gives
\[
 \partial_h\Phi=\frac1h\left(
 \frac{g(z-h)+g(z+h)}2-\frac1{2h}\int_{z-h}^{z+h}g(v)\dd v\right)<0,
\]
by the strict trapezoid inequality for a strictly concave function.
\end{proof}

\subsection{Pair blocks and the averaged Hessian}
A reflected pair of interactions naturally leads to
\begin{equation}\label{eq:J}
 J_z(h)=\widehat W(z+h)+\widehat W(z-h),\qquad
 J'_z(h)=2h\Phi(z,h).
\end{equation}
The entire interval $[z-|h|,z+|h|]$ must lie in $(0,2\pi)$.
Noncollision of two endpoints considered only modulo $2\pi$ is not
sufficient. In the applications, the two endpoints are positive proper
arc sums or their complements, so the full interval condition holds.
For two moving pairs with half-separations $0<a<b<\pi$, the four cross
interactions are
\begin{equation}\label{eq:Kblock}
 J_{b-a}(h)+J_{b+a}(h).
\end{equation}
Interactions between the two satellites of a single pair are independent
of the reflection defect. We retain them whenever derivatives with respect
to the average variables will also be needed.

In the next lemma, $t$ is the vector of reflection defects and $s$ collects
the average variables. Each $z_\nu(s)$ is the central angular separation of
one paired interaction, and each fixed vector $b_\nu$ expresses its defect
as a linear combination of the entries of $t$. The lemma derives the
reflection matrix and the concavity used in finite verification. Each later
application uses the coordinates specified for that reflection type.

\begin{lemma}[Averaged reflection matrix]\label{lem:avg}
Fix average variables $s$ and suppose that on a convex collision-free
defect domain containing zero the full restricted potential is a sum of
blocks $J_{z_\nu(s)}(b_\nu^Tt)$, defect-independent terms, and possibly
an unpaired term $\widehat W(\pi+b_0^Tt)$. Then
\begin{align}
 \nabla_t\mathcal U_s(t)&=2L(s,t)t,\label{eq:kernel}\\
 L(s,t)&=\frac12\int_0^1D_t^2\mathcal U_s(\lambda t)\dd\lambda
 \preceq L(s,0).\label{eq:average}
\end{align}
Each paired block contributes
$\Phi(z_\nu,b_\nu^Tt)b_\nu b_\nu^T$ to $L$; the unpaired term
contributes $\tfrac12\Phi(\pi,b_0^Tt)b_0b_0^T$.
If the block arguments are affine in the original gaps, then for every
fixed vector $Z$, $Z^TLZ$ is concave in those gaps. At a full angular
stationary point, $Lt=0$.
\end{lemma}
\begin{proof}
Differentiate \eqref{eq:J} term by term. For the unpaired term, oddness
of $f$ about $\pi$ gives
$f(\pi+h)=h\Phi(\pi,h)$, which explains the factor $1/2$.
For a paired block its integrated Hessian coefficient is
\[
 \frac12\int_0^1\{g(z+\lambda h)+g(z-\lambda h)\}\dd\lambda
 =\Phi(z,h).
\]
For the unpaired block, symmetry of $g$ gives the same formula with the
extra factor $1/2$. The segment to zero remains in the stated domain.
Each matrix multiplier is positive semidefinite. Thus
\eqref{eq:phi-properties} gives the matrix upper comparison, and joint
concavity gives the assertion for fixed quadratic directions. Full
stationarity implies restricted stationarity, proving $Lt=0$.
\end{proof}

\begin{remark}[What is, and is not, a Hessian]
The integral identity is a matrix identity because of the block
calculation above; it is not inferred only from its action on the single
vector $t$. The endpoint Hessian need not equal $2L(s,t)$. If $t=Tp$
for an invertible linear change of defect coordinates, the corresponding
matrix is $T^TLT$. Congruence preserves inertia and nonsingularity, not
individual eigenvalues. The same distinction applies when a common rotation is fixed before
writing the matrix for an even number of satellites.
\end{remark}

\subsection{The negative-subspace principle}
\begin{lemma}[One remaining eigenvalue]\label{lem:inertia}
Let $L$ be a real symmetric $d\times d$ matrix which is strictly negative
on a fixed $(d-1)$-dimensional subspace. If $Z^TLZ>0$ for some $Z$, then
$L$ has $d-1$ negative eigenvalues and one positive eigenvalue, and is
invertible. A negative-definite upper bound $L\preceq U\prec0$ also
proves invertibility. In either case, $Lt=0$ forces $t=0$.
\end{lemma}
\begin{proof}
By the variational characterization of eigenvalues, negativity on the
stated subspace gives at least $d-1$ negative eigenvalues. A positive
quadratic direction supplies a positive eigenvalue. All $d$ eigenvalues have therefore been accounted for, and none is zero. The assertion for a
negative-definite matrix upper bound follows directly from its quadratic form.
\end{proof}

\subsection{Scalar estimates for negativity on the chosen subspaces}
The coefficients six and seven arise in the quadratic-form estimates
for different reflection types. Both inequalities will be used. The first
provides the explicit negative margin needed for seven satellites.
\begin{lemma}[A strengthened scalar inequality]\label{lem:six}
For every $0<t<\pi$,
\begin{equation}\label{eq:six}
 g(t)+6g(\pi-t)\leq-\delta_6,\qquad \delta_6=\frac{309}{608}.
\end{equation}
Consequently
\begin{equation}\label{eq:two}
 g(t)+2g(\pi-t)\leq-\kappa,
 \qquad \kappa=\frac{3\delta_6}{7}=\frac{927}{4256}.
\end{equation}
\end{lemma}
\begin{proof}
Put $x=\sin^2(t/2)\in(0,1)$. Formula \eqref{eq:g} gives
\[
 g(t)+6g(\pi-t)=5-10x-\frac{2-x}{8x^{3/2}}
                      -\frac{6(1+x)}{8(1-x)^{3/2}}.
\]
For $x\geq1/2$, the linear term is nonpositive, and the two subtracted
fractions are at least $1/8$ and $9/8$. The expression is therefore at most
$-5/4<-\delta_6$.

For $0<x\leq1/2$, the convexity bounds
\[
 x^{-3/2}\geq20-48x,\qquad (1-x)^{-3/2}\geq1+\frac32x
\]
give
\begin{align*}
 8\{g(t)+6g(\pi-t)\}
 &\leq40-80x-(2-x)(20-48x)-6(1+x)(1+3x/2)\\
 &=-6+21x-57x^2\\
 &=-57\left(x-\frac7{38}\right)^2-\frac{309}{76}.
\end{align*}
This proves \eqref{eq:six}. Finally, eleven times \eqref{eq:six} at $t$
plus four times the same inequality at $\pi-t$, divided by $35$, gives
\eqref{eq:two}.
\end{proof}

\begin{lemma}\label{lem:seven}
For every $0<t<\pi$,
\begin{equation}\label{eq:seven}
 g(t)+7g(\pi-t)<-\frac14.
\end{equation}
Consequently
\begin{equation}\label{eq:three}
 g(t)+3g(\pi-t)<-\frac18.
\end{equation}
\end{lemma}
\begin{proof}
Put $x=\sin^2(t/2)\in(0,1)$. Formula \eqref{eq:g} gives
\[
 8\{g(t)+7g(\pi-t)\}
 =48-96x-\frac{2-x}{x^{3/2}}-\frac{7(1+x)}{(1-x)^{3/2}}.
\]
For $r=7/15$, convexity and the quadratic Taylor lower bound imply
\[
 x^{-3/2}\geq\frac{5}{2r^3}-\frac{3x}{2r^5},\qquad
 (1-x)^{-3/2}\geq1+\frac32x+\frac{15}{8}x^2.
\]
The second bound holds on all of $(0,1)$ because the third derivative
of $(1-x)^{-3/2}$ is positive. Substitution yields
\[
 8\{g(t)+7g(\pi-t)\}
 \leq -Ax^2+Bx+C-\frac{105}{8}x^3,
\]
where
\[
 A=\frac{12289023}{134456},\qquad B=\frac{783968}{16807},\qquad
 C=-\frac{2812}{343}.
\]
Completing the square, and dropping the nonpositive cubic term, gives
\[
 -Ax^2+Bx+C\leq C+\frac{B^2}{4A}
 =-\frac{464068251076}{206541609561}<-2.
\]
All constants and the polynomial expansion are checked by exact rational
arithmetic in the supplementary source. This proves \eqref{eq:seven}.
Taking $5/12$ of \eqref{eq:seven} at $t$ and $1/12$ of the same inequality
at $\pi-t$ gives \eqref{eq:three}.
\end{proof}

For use with six satellites we also record the consequence
\begin{equation}\label{eq:fivepair}
 g(t)+5g(\pi-t)\leq-\frac67\delta_6<0,
 \qquad g(t)+g(\pi-t)\leq2g(\pi/2)<0.
\end{equation}
The first follows by taking $29/35$ of \eqref{eq:six} at $t$ and
$1/35$ at $\pi-t$; the second is midpoint concavity.
These weighted consequences avoid separate proofs of the same scalar
sign in each satellite case.

\section{Collision separation and compact domains}\label{sec:compactness}
Before carrying out finite verification, we need an explicit positive lower
bound for every cyclic gap of a solution. Such bounds place all possible
solutions in compact regions away from collisions. The argument is the same
for all four satellite numbers, but each verification domain uses its own
numerical lower bound. We prove the common argument first and then give the
bounds needed in each case.

\subsection{A common propagation lemma}
The sign of \eqref{eq:f} gives $f>0$ on $(0,\pi/3)$ and
$-1<f<0$ on $(\pi/3,\pi)$. Consequently $f(t)>-1$ for
$0<t\leq5\pi/3$. Also $g(t)<g(\pi/3)=-9/4$ on $(0,\pi/3)$,
so $f$ is strictly decreasing there.

\begin{lemma}[Short-gap propagation]\label{lem:propagation}
Let $N\geq4$, put $m=N-2$, and suppose \eqref{eq:cc} holds.
If $x_i<\pi/3$ and $f(x_i)>m$, then
\begin{equation}\label{eq:propagation}
 x_{i-1}<\pi/3,\qquad f(x_{i-1})>f(x_i)/m.
\end{equation}
\end{lemma}
\begin{proof}
List the forward separations from the satellite before $x_i$ as
$a_1=x_i<a_2<\cdots<a_{N-1}=2\pi-x_{i-1}$. Their forces sum to zero.
Let $J=\{k\in\{2,\ldots,N-1\}:a_k>5\pi/3\}$ and $r=|J|$.
If $r=0$, all $m$ other terms exceed $-1$, contradicting $f(x_i)>m$.
Thus $J$ is a nonempty terminal segment. Put $y_k=2\pi-a_k$ for
$k\in J$. Its smallest member is $y_{N-1}=x_{i-1}$, and all members
are distinct and belong to $(0,\pi/3)$.

If $r<m$, the nonempty complementary sum gives
\[
 \sum_{k\in J}f(y_k)>f(x_i)-(m-r),\qquad
 \frac{f(x_i)-(m-r)}r-\frac{f(x_i)}m
 =\frac{(m-r)(f(x_i)-m)}{rm}>0.
\]
The largest term, $f(x_{i-1})$, is at least their average and hence
strictly larger than $f(x_i)/m$.
If $r=m$, the relation is instead the equality
$\sum_{k\in J}f(y_k)=f(x_i)$. There are $m\geq2$ distinct $y_k$,
so strict decrease of $f$ makes $f(x_{i-1})$ strictly larger than their
average. This separate case preserves strictness when the complementary
sum is empty.
\end{proof}

\begin{lemma}[A finite sequence of gap estimates]\label{lem:ladder}
Let $N\geq4$. Suppose $0<a_0<\cdots<a_{N-1}<\pi/3$, $\sum_j a_j<2\pi$, and
\begin{equation}\label{eq:ladderconditions}
 f(a_j)>(N-2)f(a_{j+1})\quad(0\leq j<N-1),
 \qquad f(a_{N-2})>N-2.
\end{equation}
Then every angular solution has $x_i>a_0$ for every $i$.
\end{lemma}
\begin{proof}
If $x_i\leq a_0$, monotonicity and Lemma~\ref{lem:propagation} yield
successively $x_{i-j}<a_j$, $1\leq j<N$. Each application is legitimate
because the force at the current gap is above $f(a_j)\geq f(a_{N-2})$
when $j\leq N-2$. Summing the resulting inequalities contradicts
$\sum_i x_i=2\pi$.
\end{proof}

\subsection{Explicit lower bounds for the verification domains}
For $N=5,6,7$, use the following increasing sequences of rational angles,
each satisfying
\eqref{eq:ladderconditions}:
\begin{align}
 N=5:&\quad (1/10,9/50,8/25,21/40,4/5),\label{eq:ladder5}\\
 N=6:&\quad (1/40,11/200,23/200,6/25,23/50,3/4),\label{eq:ladder6}\\
 N=7:&\quad (1/160,3/200,7/200,2/25,9/50,2/5,4/5).
 \label{eq:ladder7}
\end{align}
Their sums are respectively $77/40$, $329/200$, and $1213/800$.
The eighteen comparisons are finite exact-rational checks. For
$0<t\leq4/5$, the alternating sine polynomials
\[
 s_-(t)=\sum_{k=0}^{5}\frac{(-1)^kt^{2k+1}}{(2k+1)!},\qquad
 s_+(t)=\sum_{k=0}^{4}\frac{(-1)^kt^{2k+1}}{(2k+1)!}
\]
satisfy $s_-(t)<\sin t<s_+(t)$. The successive absolute-term ratio is
at most $(4/5)^2/6<1$. At each angle in these sequences, the verification program first checks
$0<s_-(t/2)<s_+(t/2)$ and $(8s_+(t/2)^3)^{-1}-1>0$.
It can then bound $f(t)$ below and above by
\[
 s_-(t)\bigl((8s_+(t/2)^3)^{-1}-1\bigr),\qquad
 s_+(t)\bigl((8s_-(t/2)^3)^{-1}-1\bigr),
\]
respectively. Strict rational cross-multiplication verifies every
comparison. These lower bounds for the gaps do not require a search for
solutions of the angular equations.

For eight satellites, a single analytic estimate verifies all seven
comparisons between consecutive angles in the required sequence.

\begin{lemma}[Elementary small-angle bounds]\label{lem:smallangle}
For $0<t\leq1$,
\begin{equation}\label{eq:smallangle}
 t^{-2}-\frac18-t\leq f(t)\leq t^{-2}-\frac56t.
\end{equation}
For $0<t\leq2/5$, it follows that
\begin{equation}\label{eq:uniformladder}
 f(t)-6f(5t/2)\geq\frac1{25t^2}-\frac18+\frac{23}2t>\frac18.
\end{equation}
\end{lemma}
\begin{proof}
Put $u=t/2$. Since $f(t)=\cos u/(4\sin^2u)-\sin t$,
$\cos u\geq1-u^2/2>0$, $\sin u\leq u$, and $\sin t\leq t$ give
the lower bound. Integrating the chord of the concave cosine on $[0,u]$
gives $\sin u/u\geq(1+\cos u)/2$. Its square is at least $\cos u$,
so $\cos u/(4\sin^2u)\leq t^{-2}$. Also
$\sin t\geq t-t^3/6\geq5t/6$, proving the upper bound.
Apply these bounds at $t$ and $5t/2\leq1$. Finally
$1/(25t^2)\geq1/4$ proves the strict last inequality.
\end{proof}

\begin{proposition}[Explicit gap bounds for the verification domains]\label{prop:collars}
Every collision-free angular solution satisfies the corresponding bound
in the following table.
\begin{center}
\begin{tabular}{ccc}
\toprule
$N$ & Gap bound in radians & Bound used for finite verification\\
\midrule
$5$ & $x_i>1/10$ & Classification uses $0<\delta<1/10$ in radians\\
$6$ & $x_i>1/40$ & $R_i>1/128$\\
$7$ & $x_i>1/160$ & $R_i>1/512$\\
$8$ & $x_i>1/640$ & $R_i>1/2048$\\
\bottomrule
\end{tabular}
\end{center}
\end{proposition}
\begin{proof}
The first three bounds follow from Lemma~\ref{lem:ladder} and
\eqref{eq:ladder5}--\eqref{eq:ladder7}. For $N=8$ set
$a_j=640^{-1}(5/2)^j$, $0\leq j\leq7$. Then
\[
 a_6=3125/8192<5/13<2/5,\qquad a_7=15625/16384<1<\pi/3.
\]
Equation~\eqref{eq:uniformladder} supplies the seven ladder comparisons.
Since $t^{-2}-1/8-t$ is decreasing,
\[
 f(a_6)>(13/5)^2-1/8-5/13=6+651/2600>6.
\]
The geometric sum $\sum_{j=0}^7a_j<2<2\pi$ completes the hypotheses
of Lemma~\ref{lem:ladder}. Finally Machin's identity gives
$\pi/4=4\arctan(1/5)-\arctan(1/239)<4/5$; hence $\pi<16/5$,
which yields all three displayed normalized bounds.
\end{proof}

\begin{remark}[Why different gap bounds are needed]
The different lower bounds are needed for the different initial
verification domains. The weaker common inequality $x_i>1/640$ alone
would not put every six-satellite solution in $R_i\geq1/128$, or every
seven-satellite solution in $R_i\geq1/512$. The common propagation argument must therefore supply the bound required
by each specified verification domain.
\end{remark}

\section{Rigorous finite verification on compact regions}\label{sec:certificates}
The analytic argument reduces the symmetry problem to inequalities on compact
regions. We subdivide an initial region into smaller closed regions and verify
one of the criteria below on every final subregion. The criteria either prove
that a candidate reflection matrix is nonsingular or exclude a solution of the
original force equations.

For each subregion, the stored verification data specify the criterion, the
chosen matrix or force combination, and the bounds to be checked. We call this
its \emph{verification record}. A separate verification program reconstructs
the regions and checks the inequalities. The number of subdivisions measures
the size of the computation; its validity depends on complete coverage and on
the stated mathematical implications.

\subsection{Closed verification regions and complete coverage}
For a rational box $B=\prod_i[\ell_i,u_i]$, put
\begin{equation}\label{eq:boxslice}
 P_B=B\cap\left\{\sum_iR_i=2\right\}.
\end{equation}
The set $P_B$ is the intersection of the box with the fixed-sum
hyperplane; we refer to it as a \emph{fixed-sum section}. Some initial
domains also impose ordering inequalities to select representatives under
cyclic relabeling and reversal. An estimate on all of $P_B$ remains valid
on the subset satisfying these additional inequalities. The reference point
or vertices used to prove that estimate need not satisfy the additional
inequalities themselves.

\begin{lemma}[Vertices and subdivision by closed inequalities]\label{lem:covergeometry}
Every vertex of a nonempty $P_B$ has at most one coordinate strictly
between its bounds. All vertices are therefore found by fixing all but
one coordinates at bounds and solving the sum equation, retaining feasible
points and removing duplicates. This procedure also covers lower-dimensional sections
and sections consisting of a single point. Tightening a bound by a necessary inequality preserves the
represented set. Two closed cuts $R_j\leq c$ and $R_j\geq c$ cover the
original region, including the common boundary.
\end{lemma}
\begin{proof}
If two coordinates of a vertex were strictly inside their bounds, small
opposite changes in them would preserve the sum and the box, contradicting
extremality. The enumeration contains every remaining possibility.
The tightening assertion is a set inclusion in both directions on the
constrained set. The assertion about two closed cuts is their union
identity. The argument allows equality in the ordering constraints and does not
require a unique representative of each configuration.
\end{proof}

In particular, valid sum tightenings include
\[
 \ell_i\leftarrow\max\{\ell_i,2-\sum_{j\ne i}u_j\},\qquad
 u_i\leftarrow\min\{u_i,2-\sum_{j\ne i}\ell_j\}.
\]
They can be applied successively using already valid bounds. Only strict
inconsistencies, such as $\ell_i>u_i$ or $\sum\ell_i>2$, prove
emptiness. Equality does not prove emptiness. For bookkeeping, the recursive subdivisions are stored as a finite tree: its
root is the initial region, its children are the regions produced by one
split, and its terminal nodes represent the final subregions. We use the tree
only to record the coverage. The verification program must check the initial
parameters, each subdivision relation, that every recorded node is reached
from the initial region, and the criterion and inequalities for every final
subregion. A missing, reused, or unreachable record cannot establish that any
part of the domain has been verified. In the tables below, X denotes a region
proved empty; P, N, TN, F, and C denote the five criteria defined next.

\subsection{Positive quadratic directions and negative-definite upper bounds}
In each case the matrix is a sum of fixed positive semidefinite rank-one
matrices with scalar coefficients:
\begin{equation}\label{eq:rankone}
 L(R)=\sum_\nu c_\nu\Phi(\pi z_\nu(R),\pi h_\nu(R))v_\nu v_\nu^T,
 \qquad c_\nu>0,
\end{equation}
with affine normalized arguments. The value $c_\nu=1/2$ is retained
for the unpaired antipodal block.

\begin{lemma}[Matrix criteria P and N]\label{lem:PN}
Suppose $L$ has a known negative subspace of codimension one.
Criterion \textup{P} uses a fixed nonzero rational direction
$w$ and strictly positive lower bounds for $w^TLw$ at every vertex of $P_B$.
These bounds prove that $L$ is nonsingular throughout $P_B$.
Criterion \textup{N} uses a matrix upper bound $L(R)\preceq U(B)\prec0$
valid on the whole fixed-sum section, and gives the same conclusion.
\end{lemma}
\begin{proof}
The scalar form $w^TL(R)w$ is concave by Lemma~\ref{lem:avg}, so its
minimum on a polytope is bounded below by the least of its vertex values.
A positive bound and Lemma~\ref{lem:inertia} prove the P assertion.
The N assertion is the other part of that lemma.
\end{proof}

The two criteria require different bounds: P requires a lower bound,
whereas N requires a matrix upper bound. Negative scalar
values at the vertices of a concave function do \emph{not} prove
negativity in the interior. For N, one may first use $L(s,t)\preceq
L(s,0)$ and bound each $g$ coefficient above over its argument interval.
The maximum of $g$ occurs at the point nearest $\pi$. Summing scalar
upper bounds with the rank-one multipliers gives the required matrix
upper bound. Exact leading-principal-minor signs, or an exact congruence
factorization, then verify strict negative definiteness.

The recorded P verifications use two ways to obtain lower bounds. A concave
function lies above its chords. Thus a rational piecewise-linear lower
model $\ell(r)<g(\pi r)$ gives
\begin{equation}\label{eq:chordlower}
 \frac1{b-a}\int_a^b\ell(r)\dd r
 <\Phi\left(\frac{\pi(a+b)}2,\frac{\pi(b-a)}2\right)
 \quad(a<b),
\end{equation}
with the value $\ell(a)$ when $a=b$. The integrals of a rational
piecewise-linear model are exact rational numbers. Alternatively,
composite trapezoids in the integral \eqref{eq:phi} give lower bounds,
and composite midpoint values give upper bounds, because the integrand
is concave. For $m$ panels the trapezoid endpoint weights are $1/(2m)$,
its interior weights are $1/m$, and midpoint weights are $1/m$.
Both formulas include $h=0$.

For completeness, a rational sine lower bound is converted to a $g$
lower bound using
\begin{equation}\label{eq:G}
 g(t)=G(\sin(t/2)),\qquad
 G(s)=-\frac1{4s^3}+\frac1{8s}-1+2s^2,
 \quad 8s^4G'(s)=32s^5-s^2+6>0.
\end{equation}
The checker must establish positivity of any lower sine denominator
before division. Argument reduction, rational bounds on $\pi$, Taylor
remainders, and outward rounding are part of the verification. Outward
rounding means rounding lower interval endpoints downwards and upper
endpoints upwards, so that the resulting intervals still contain the exact
values. These steps cannot be replaced by floating-point assumptions.

\subsection{Excluding solutions using the original forces}
A force estimate excludes a solution of the full angular equations
whenever a force or a fixed linear combination of forces has a strict sign
throughout the region. This criterion does not require a nonsingular
reflection matrix.
For $i<j$ set $A_{ij}(R)=\sum_{k=i}^{j-1}R_k$. Then
\begin{align}
 F_i(R)&=\sum_{j>i}f(\pi A_{ij}(R))-
                  \sum_{j<i}f(\pi A_{ji}(R)),\label{eq:fullforce}\\
 H_w(R)&=\sum_iw_iF_i(R)
       =\sum_{i<j}(w_i-w_j)f(\pi A_{ij}(R)).\label{eq:H}
\end{align}
At a full stationary point every $H_w$ vanishes, regardless of the
process used to propose $w$. Adding a constant to all weights has no
effect. Criterion F proves that a single $F_i$ has a strict sign; criterion C
proves that a selected $H_w$ has a strict sign. Both exclude solutions
throughout the region; neither establishes existence or uniqueness of a zero.

Simple force intervals use exact arc bounds. For an arc index set $I$,
\[
 \max\left\{\sum_{i\in I}\ell_i,2-\sum_{i\notin I}u_i\right\}
 \leq\sum_{i\in I}R_i\leq
 \min\left\{\sum_{i\in I}u_i,2-\sum_{i\notin I}\ell_i\right\}.
\]
Direct interval evaluation can give wide bounds when large terms cancel.
The next estimate retains more of this cancellation by expanding the full
force combination about a feasible reference point.

\begin{lemma}[Force bounds about a point in the fixed-sum section]\label{lem:centered}
Choose $c\in P_B$ rational, and use $R_0,\ldots,R_{N-2}$ as local
coordinates, with $R_{N-1}=2-\sum_{k<N-1}R_k$. Let $[h_c]$ enclose
$H_w(c)$ and let $J_k$ enclose
\begin{equation}\label{eq:forcejac}
 \partial_{R_k}H_w(R)=
 \pi\sum_{i\leq k<j}(w_i-w_j)g(\pi A_{ij}(R)),\quad 0\leq k<N-1,
\end{equation}
throughout $P_B$. Then
\begin{equation}\label{eq:centered}
 H_w(R)\in[h_c]+\sum_{k<N-1}J_k([\ell_k,u_k]-c_k).
\end{equation}
More sharply, choose rational $q_k$ and radii $e_k\geq0$ with
$J_k\subset[q_k-e_k,q_k+e_k]$, and set $q_{N-1}=0$. With
$d_k=\max\{c_k-\ell_k,u_k-c_k\}$ and $E=\sum_{k<N-1}e_kd_k$,
\begin{equation}\label{eq:supportcentered}
 H_w(P_B)\subseteq[h_c]+\left[
 \min_{R\in P_B}q\cdot(R-c),\ 
 \max_{R\in P_B}q\cdot(R-c)\right]+[-E,E].
\end{equation}
If either enclosing interval excludes zero, criterion C applies.
\end{lemma}
\begin{proof}
The segment from $c$ to $R$ lies in the convex fixed-sum section. Integrate the
gradient along that segment to obtain \eqref{eq:centered}. Subtract
$q$ from the gradient before integration and bound each residual by
$e_kd_k$ to obtain \eqref{eq:supportcentered}. The displayed affine
extrema are attained at slice vertices and hence can be computed exactly.
Formula~\eqref{eq:forcejac} follows from the original pair sum; the
factor $\pi$ is essential.
\end{proof}

A rational feasible center is
$c=\ell+\alpha(u-\ell)$ with
$\alpha=(2-\sum_i\ell_i)/\sum_i(u_i-\ell_i)\in[0,1]$.
If the denominator is zero, a nonempty section consists only of the point
$c=\ell$, where the function is evaluated directly. The construction
therefore requires no division by zero.

\subsection{An affine matrix upper bound: criterion TN}
For eight satellites, a variant of criterion N uses affine tangent bounds
in place of constant scalar upper bounds. Put $\mathfrak g(r)=g(\pi r)$, so
$\mathfrak g'(r)=\pi g'(\pi r)$. Concavity gives
\begin{equation}\label{eq:tangent}
 \mathfrak g(a(R))\leq
 \mathfrak g(\alpha)+\mathfrak g'(\alpha)(a(R)-\alpha),
 \qquad \alpha=a(R_*),
\end{equation}
for a fixed feasible rational reference point $R_*$. Outward interval
bounds on the value and slope give a rational affine upper bound for the
right-hand side. Since the difference of two affine expressions is bilinear
in the argument and the slope uncertainty, checking both argument
endpoints and both slope endpoints suffices. Dropping the factor $\pi$
or its sign after argument reflection would invalidate the estimate.

\begin{lemma}[Negative-definiteness criterion TN using tangent bounds]\label{lem:TN}
Suppose the scalar bounds \eqref{eq:tangent}, summed with the positive
rank-one multipliers, yield $L(R)\preceq A(R)$ for an affine symmetric
matrix $A$. If for some rational $\gamma>0$,
$A(v)+\gamma I\prec0$ at every vertex $v$ of $P_B$, then
$L(R)\prec-\gamma I$ throughout $P_B$.
\end{lemma}
\begin{proof}
Every point of the slice is a convex combination of its vertices, and
$A$ is affine. Negative definite matrices form a convex cone. The same
strict margin $\gamma$ therefore holds at every point of the slice.
\end{proof}
The vertex test applies to the affine matrix \emph{upper bound}.
Negativity at the vertices of the original concave coefficients would not
suffice. Each upper bound must be verified on the actual final box obtained
from all preceding splits and bound tightenings. A check on a different box
used when choosing the proposed bound does not suffice.

\subsection{The final covering implication}
\begin{proposition}[From a finite cover to reflection symmetry]\label{prop:coverlogic}
Let a closed initial verification domain contain a representative of every
solution not already treated analytically. Suppose a finite recursive
subdivision covers this domain using valid closed splits and bound
tightenings that preserve every admissible point. If every final subregion
is empty or satisfies one of the criteria P, N, TN, F, or C above, then every
solution represented by the initial domain has a reflection axis.
\end{proposition}
\begin{proof}
At each split, at least one of the two closed subregions contains the chosen
solution, and each valid bound tightening preserves it. Since there are only
finitely many subdivisions, the solution belongs to a final subregion.
That subregion cannot be empty. Criteria F and C would contradict the
original force equations. Criteria P, N, and TN give an invertible candidate
reflection matrix, so the exact equation $Lt=0$ forces the corresponding
reflection defect to vanish. Relabeling or reversing the cyclic order
preserves the existence of the resulting geometric reflection.
\end{proof}

\section{Reflection symmetry for five satellites}\label{sec:n5}
For five satellites the reflection matrix has dimension two. If two gaps
are equal, the configuration is already reflection-symmetric. If all gaps
are distinct, ordering them reduces the proof to three possible cyclic
arrangements. Only this case requires consideration of every ordering of
the gaps.

\subsection{Reflection equations and a common negative direction}
Choose a candidate fixed satellite $P_0$ and list the successive gaps as
\begin{equation}\label{n5:eq:order}
 (d,c,e,b,a)=(u+p,v+q,e,v-q,u-p),\qquad
 e=2\pi-2u-2v>0.
\end{equation}
Here $u>|p|$, $v>|q|$, and $u+v<\pi$. The five candidates are obtained
by cyclic shifts. The candidate is a geometric reflection precisely when
$p=q=0$. Set
\begin{equation}\label{n5:eq:coeff}
 A=\Phi(u,p),\quad B=\Phi(u+v,p+q),\quad
 C=\Phi(v,q),\quad D=\Phi(2u+v,q).
\end{equation}
The last coefficient also equals $\Phi(v+e,q)$.
The moving pairs have cumulative defects $t_1=p$, $t_2=p+q$.
Apart from two interactions within individual reflected pairs, their potential is
$J_u(p)+J_{u+v}(p+q)+J_v(q)+J_{2u+v}(q)$.
Lemma~\ref{lem:avg} therefore gives
\begin{equation}\label{n5:eq:matrix}
 M\binom pq=0,\qquad
 M=\begin{pmatrix}A+B&B\\B&B+C+D\end{pmatrix},
\end{equation}
at every full stationary point. In particular,
\begin{equation}\label{n5:eq:twoeq}
 Ap+B(p+q)=0,\qquad B(p+q)+(C+D)q=0.
\end{equation}

\begin{lemma}[A fixed negative direction]\label{n5:lem:negative}
Every collision-free candidate satisfies
\begin{equation}\label{n5:eq:negative}
 (1,-1)M(1,-1)^T=A+C+D<0.
\end{equation}
\end{lemma}
\begin{proof}
Use $\Phi(z,h)\leq g(z)$ and reflection of $g$:
\[
 A+C+D\leq g(u)+g(v)+g(2u+v)
 =g(u)+g(v)+g(2\pi-2u-v)
 <g(u)+2g(\pi-u)<0.
\]
The strict midpoint inequality is valid because $u+v<\pi$; the last
bound is \eqref{eq:two}. Thus the same scalar estimate used for seven
satellites supplies this negative direction without a separate
one-variable proof.
\end{proof}

The large-gap region can also be closed directly in these coordinates.
If the chosen gap $e\geq\pi$, then $u+v\leq\pi/2$. Hence
$A\leq g(u)<0$, $B\leq g(u+v)<0$, and
$C+D\leq g(v)+g(2u+v)\leq2g(u+v)<0$.
The quadratic form
\[
 A X^2+B(X+Y)^2+(C+D)Y^2
\]
is negative definite, so \eqref{n5:eq:matrix} forces reflection.
Thus the remaining case has all five gaps at most $\pi$.
The same argument includes the boundary $e=\pi$.

\subsection{Repeated-gap rigidity}
Among five cyclic gap positions, two distinct positions are separated by
one or two steps around the cycle. Thus any equality between two gaps can
be placed in one of the two forms considered in the next lemma.

\begin{lemma}\label{n5:lem:repeatlocal}
At an angular stationary point, $a=d$ implies $b=c$, and $b=c$ implies
$a=d$, in the coordinates \eqref{n5:eq:order}.
\end{lemma}
\begin{proof}
If $p=0$ and $q\ne0$, \eqref{n5:eq:twoeq} gives $B=0$ and $C+D=0$.
For the common half-width $q$, the center $u+v$ is the midpoint of
$v$ and $2u+v$. Strict concavity in the center gives
$B>(C+D)/2=0$, a contradiction. Thus $q=0$.
If $q=0$ and $p\ne0$, the same equations give $B=A=0$.
But $u<u+v<\pi$ and strict increase of $\Phi(z,p)$ below $\pi$
give $A<B$, again a contradiction.
\end{proof}

\begin{proposition}[Repeated gaps force a reflection]\label{n5:prop:repeated}
If any two gaps of a five-satellite angular solution agree, the
configuration has a reflection axis.
\end{proposition}
\begin{proof}
The cyclic distance between two distinct gap positions in a five-cycle
is one or two. A cyclic shift and, if necessary, reversal put their
equality in the form $a=d$ or $b=c$. Lemma~\ref{n5:lem:repeatlocal}
then gives both equalities, which are the reflection condition.
\end{proof}

\subsection{A sign law and pair-sum positivity}
Suppose for contradiction that a solution is not reflection-symmetric.
For every candidate, $p,q,p+q$ are nonzero. Indeed, either of the first
two vanishing would force reflection. If $p+q=0$, then
\eqref{n5:eq:twoeq} would give $A=C+D=0$, contradicting
\eqref{n5:eq:negative}. Dividing only after these checks yields
\begin{equation}\label{n5:eq:signlaw}
 A+C+D=-B\frac{(p+q)^2}{pq}<0,
 \qquad\sgn B=\sgn(pq).
\end{equation}
To use this law, we need a geometric condition guaranteeing $B>0$.

\begin{lemma}[Pair-sum positivity]\label{n5:lem:pairsum}
If $0<e<2\pi/3$ and $\min\{d+c,a+b\}>\pi/3$, then $B>0$.
\end{lemma}
\begin{proof}
Let $x=\min\{d+c,a+b\}$ and $y=\max\{d+c,a+b\}$. Then
$x+y=2\pi-e$. If $x=y$, the coefficient is
$g(\pi-e/2)>g(2\pi/3)=1/2-5\sqrt3/36>0$.
Assume $x<y$. If $x+e\leq\pi$, then $f(x)<0$ and
$f(y)=-f(x+e)\geq0$, proving the positive divided difference.
If $x+e>\pi$, then $\pi/3<x<y<\pi$. Put
$m=(x+y)/2=\pi-e/2>2\pi/3$ and $h=(y-x)/2<\pi/3$.
Lemma~\ref{lem:phi} gives
\[
 B=\Phi(m,h)>\Phi(2\pi/3,h)>\Phi(2\pi/3,\pi/3)
 =\frac{f(\pi)-f(\pi/3)}{2\pi/3}=0.
\]
All intermediate divided differences lie in their natural domains.
\end{proof}

Write the distinct gaps in increasing order as
$0<r_1<r_2<r_3<r_4<r_5\leq\pi$, with sum $2\pi$. Then
\begin{equation}\label{n5:eq:rankbounds}
 r_3<2\pi/3,\qquad r_1+r_4>\pi/3.
\end{equation}
For the first inequality, otherwise $r_3+r_4+r_5>2\pi$.
For the second, otherwise
$r_1+r_2+r_3+r_4<r_1+3r_4\leq\pi-2r_1<\pi$, forcing $r_5>\pi$.

There are $5!/10=12$ cyclic arrangements of the ranked gaps, up to reversal.
An arrangement is written as its sequence of ranks. The table gives one
candidate axis for each arrangement; its five entries are the ranks of
$(d,c,e,b,a)$.
\begin{center}
\begin{tabular}{ccc@{\qquad}ccc}
\toprule
Order&Candidate&$\sgn pq$&Order&Candidate&$\sgn pq$\\
\midrule
12345&51234&$-$&12354&41235&$+$\\
12435&24351&$-$&12453&53124&$+$\\
12534&34125&$-$&12543&43125&$-$\\
13245&51324&$-$&13254&41325&$+$\\
13425&34251&$-$&13524&24135&$-$\\
14235&14235&$-$&14325&14325&$-$\\
\bottomrule
\end{tabular}
\end{center}
In every row $e$ has rank at most three, and each of $d+c,a+b$ contains
a gap of rank at least four. Equation~\eqref{n5:eq:rankbounds} and
Lemma~\ref{n5:lem:pairsum} therefore imply $B>0$. The nine negative
rows contradict \eqref{n5:eq:signlaw}. Only the arrangements
$12354$, $12453$, and $13254$ remain.

\subsection{The three remaining gap arrangements}
For a fixed candidate and a real number $\tau$, put
\begin{equation}\label{n5:eq:Qtau}
 Q_\tau=(1,\tau)M(1,\tau)^T
       =A+(1+\tau)^2B+\tau^2(C+D).
\end{equation}
The form $Q_{-1}<0$ is already known. It is enough to find another
fixed direction with positive form. Each $Q_\tau$ is concave in the
gaps by Lemma~\ref{lem:avg}. The definition of $\Phi$ retains its sign and allows either sign of
$p+q$.

Normalize the ranked gaps by $R_i=r_i/\pi$ and use the closed polytope
\begin{equation}\label{n5:eq:rankpoly}
 0\leq R_1\leq R_2\leq R_3\leq R_4\leq R_5\leq1,
 \qquad\sum_{i=1}^5R_i=2.
\end{equation}
For $12354$, choose
$L_1=(r_5,r_4,r_1,r_2,r_3)$ and
$L_2=(r_4,r_1,r_2,r_3,r_5)$.
For $13254$, choose
$L_1=(r_5,r_4,r_1,r_3,r_2)$ and
$L_2=(r_4,r_1,r_3,r_2,r_5)$.
Again each displayed tuple is $(d,c,e,b,a)$.
For each arrangement, the following closed subregions cover the ranked-gap
polytope. They are defined by two successive splits, with equality included
on both sides of each split.

\begin{table}[htbp]
\centering\small
\begin{tabular}{clcccr}
\toprule
Region&Additional inequalities&Axis&$\tau$&Vertices&Lower margin\\
\midrule
$P_{11}$&$R_2\leq1/3,\ R_1\leq1/5$&$L_2$&$0$&14&$28377/364000$\\
$P_{12}$&$R_2\leq1/3,\ R_1\geq1/5$&$L_2$&$1/5$&12&$213593/3500000$\\
$P_{13}$&$R_2\geq1/3,\ R_5\leq2/3$&$L_1$&$4/5$&9&$17989/84000$\\
$P_{14}$&$R_2\geq1/3,\ R_5\geq2/3$&$L_1$&$1/2$&5&$22259/67200$\\
\midrule
$P_{21}$&$R_2\leq2/7,\ R_1\leq1/5$&$L_2$&$0$&16&$17/1350$\\
$P_{22}$&$R_2\leq2/7,\ R_1\geq1/5$&$L_2$&$1/5$&12&$31/12500$\\
$P_{23}$&$R_2\geq2/7,\ R_5\leq2/3$&$L_1$&$4/5$&13&$597871/6000000$\\
$P_{24}$&$R_2\geq2/7,\ R_5\geq2/3$&$L_1$&$1/2$&14&$4217/480000$\\
\bottomrule
\end{tabular}
\caption{Verified subregions for $12354$ (first four rows) and $13254$
(last four rows). Each margin is a lower bound for the indicated fixed
quadratic form at every vertex; it is not a lower bound for all eigenvalues.}\label{n5:tab:cells}
\end{table}

For $12453$, choose $(d,c,e,b,a)=(r_5,r_3,r_1,r_2,r_4)$ and write
\[
 U=\frac{a+d}{2\pi},\quad V=\frac{b+c}{2\pi},\quad
 P=\frac{d-a}{2\pi},\quad Q=\frac{c-b}{2\pi}.
\]
The weak ordering $0\leq e\leq b\leq c\leq a\leq d\leq\pi$ is
exactly
\begin{equation}\label{n5:eq:12453}
 \begin{gathered}
 U,V,P,Q\geq0,\quad U+V\leq1,\quad U+P\leq1,\\
 2U+3V-Q\geq2,\qquad V+P+Q\leq U.
 \end{gathered}
\end{equation}
Indeed the normalized gaps are
$e/\pi=2-2U-2V$, $b/\pi=V-Q$, $c/\pi=V+Q$,
$a/\pi=U-P$, and $d/\pi=U+P$.
Split at $U=1/2$. The upper half has nine vertices. For $Q_0=A+B$,
eight have lower bound $72097/472500>0$; the last is
\[
 (U,V,P,Q)=(1/2,1/3,1/6,0),\qquad Q_0=0.
\]
At that point the two divided differences telescope to
$f(\pi)-f(\pi/3)=0$. It lies on the boundary where two gaps agree, namely $b=c$.
The lower half has five vertices and uses $Q_1$, whose lower bound is
$793/210000>0$ at every vertex.

\begin{lemma}[Extending the required quadratic form to collision boundaries]\label{n5:lem:activeextension}
For each closed subregion above, substitute its fixed value of $\tau$
in \eqref{n5:eq:Qtau} and delete every term with zero multiplier.
The sum of the remaining terms extends continuously and concavely to
the whole subregion. Each of its divided-difference endpoints belongs to
$[\pi/5,12\pi/7]$.
\end{lemma}
\begin{proof}
Each remaining endpoint is affine in the gap variables. Exact vertex
enumeration of the ten stated polytopes bounds every such normalized
endpoint between $1/5$ and $12/7$. The same bounds hold on their convex
hulls. Thus each term with nonzero multiplier stays in the domain of $\Phi$, and
Lemma~\ref{lem:phi} proves continuity and concavity of their nonnegative
linear combination. In the collision-free part this expression equals
the original quadratic form. We use this continuous extension for all
boundary vertex values below; the full reflection matrix itself need
not extend to a boundary where a collision occurs.
\end{proof}

The next proposition states the finite arithmetic checks used in the
argument for these three arrangements. They apply the same positivity
principle as criterion P, with one explicitly identified zero value on
the boundary.

\begin{proposition}[Exact inequalities for the three remaining arrangements]\label{n5:prop:vertices}
The eight subregions in Table~\ref{n5:tab:cells} have the listed vertices and
positive rational lower margins. The two half-polytopes in
\eqref{n5:eq:12453} have the margins and single zero identity just
specified. Only nonzero-multiplier divided differences are evaluated at boundary
vertices, using Lemma~\ref{n5:lem:activeextension}. In total there are 109 vertex
checks: 108 strict positive comparisons and one exact identity.
\end{proposition}
\begin{proof}[Exact verification of the finite inequalities]
An independent verification program reconstructs every vertex by exact linear
algebra from the stated half-spaces and the sum constraint. It deletes
zero-multiplier terms before evaluation and first checks the endpoint bounds for all terms with nonzero multiplier in Lemma~\ref{n5:lem:activeextension}.
It encloses $\pi$ by Machin's identity and evaluates sine and cosine by
rational Horner polynomials with explicit Taylor remainders, using
integer interval endpoints at a fixed scale, with outward rounding. For nonzero half-width it
combines the endpoint-force coefficients before division by $\pi$;
for zero half-width it evaluates $g$ directly. All denominators are
certified to avoid zero. The 108 strict comparisons exceed the printed
rational margins. The remaining value is the exact identity
$3\{f(\pi)-f(\pi/3)\}/\pi=0$.
The complete input geometry, reconstructed vertices, and directed
bounds are specified by \texttt{n5\_chamber\_exact\_referee.py}.
The finite verification consists of 108 strict positive comparisons and
one exact identity, checked with integer interval arithmetic at 160 and
224 bits. Both precisions give the stated inequalities.
\end{proof}

Concavity now gives $Q_\tau>0$ throughout the first eight subregions and
throughout the lower half of $12453$. In the upper half, a zero lower
bound can occur only at a convex combination of zero vertices. There is
only one such vertex, and its gaps are not all distinct. Thus $Q_0>0$
at every point of this region whose gaps are strictly ordered. Together with
$Q_{-1}<0$, this makes the candidate matrix invertible and contradicts
a nonzero reflection defect.

\begin{proposition}[Five-satellite symmetry]\label{n5:prop:symmetry}
Every collision-free five-satellite angular solution has a reflection axis.
\end{proposition}
\begin{proof}
A gap at least $\pi$ is covered by the direct negative-matrix argument
above. Otherwise all gaps are at most $\pi$. A repeated gap forces
reflection by Proposition~\ref{n5:prop:repeated}. If all gaps are
distinct, the sign law excludes nine cyclic arrangements, and the exact
inequality checks exclude the remaining three. The nonsymmetric
case is therefore impossible.
\end{proof}

\section{Six satellites: two reflection types}\label{sec:n6}
A reflection of six distinct circle points fixes either no satellite or
two antipodal satellites. In the first case its intersections with the
circle bisect opposite gaps. In the second, the fixed satellites are
opposite in cyclic order. There are three pairings of each type.
The constructions below use differences of individual paired gaps, as in
the recorded six-satellite verification. The stored integer directions
refer to these coordinates and must be evaluated in the same coordinates.
\subsection{Candidate axes through two opposite gaps}

Choose two opposite gaps and write the cyclic sequence as
\begin{equation}\label{n6:eq:edge-order}
 (d_1,d_2,e,a_2,a_1,h).
\end{equation}
Introduce averages and reflection defects
\begin{equation}\label{n6:eq:edge-variables}
 u=\frac{d_1+a_1}{2},\qquad p=\frac{d_1-a_1}{2},
 \qquad
 v=\frac{d_2+a_2}{2},\qquad q=\frac{d_2-a_2}{2}.
\end{equation}
Thus
\begin{equation}\label{n6:eq:edge-param}
 (d_1,d_2,e,a_2,a_1,h)
 =(u+p,v+q,e,v-q,u-p,h)
\end{equation}
and
\begin{equation}\label{n6:eq:edge-sum}
 2u+2v+e+h=2\pi.
\end{equation}
This pairing, called the \emph{edge--edge type}, gives an axis through
the midpoints of the two selected gaps exactly
when
\begin{equation}\label{n6:eq:edge-zero}
 p=q=0.
\end{equation}
Indeed, in that case the gaps read $(u,v,e,v,u,h)$, and the midpoints of $e$
and $h$ are antipodal.

Define
\begin{align}
 A&=\Phi(u,p),&D&=\Phi(u+h,p),\nonumber\\
 B&=\Phi(u+v,p+q),&C&=\Phi(u+v+e,p+q),\nonumber\\
 E&=\Phi(v,q),&F&=\Phi(v+e,q).\label{n6:eq:edge-coefficients}
\end{align}

\begin{proposition}[Edge reflection system]\label{n6:prop:edge-system}
Every central configuration satisfies
\begin{equation}\label{n6:eq:edge-system-vector}
 M_E\binom pq=0,
\end{equation}
where
\begin{equation}\label{n6:eq:edge-matrix}
 M_E=
 \begin{pmatrix}
 A+D+B+C&B+C\\
 B+C&B+C+E+F
 \end{pmatrix}.
\end{equation}
Equivalently,
\begin{align}
 p(A+D)+(p+q)(B+C)&=0,\label{n6:eq:edge-eq1}\\
 (p+q)(B+C)+q(E+F)&=0.\label{n6:eq:edge-eq2}
\end{align}
The associated quadratic form is
\begin{equation}\label{n6:eq:edge-Q}
 Q_E(X,Y)=(A+D)X^2+(B+C)(X+Y)^2+(E+F)Y^2.
\end{equation}
\end{proposition}

\begin{proof}
Put the midpoint of the gap $h$ at angle zero.  The six angles can then be
represented modulo $2\pi$ as the three paired sets
\begin{equation}\label{n6:eq:edge-angle-pairs}
 \left\{\pm\frac h2\right\},\qquad
 \left\{p\pm\left(\frac h2+u\right)\right\},\qquad
 \left\{p+q\pm\left(\frac h2+u+v\right)\right\}.
\end{equation}
The within-pair interactions do not depend on $p,q$.  Applying
\eqref{eq:Kblock} to the three cross-pair blocks and using
$J_{2\pi-z}=J_z$ gives
\begin{equation}\label{n6:eq:edge-potential-restricted}
 \begin{split}
 \cU_E(p,q)={}&J_u(p)+J_{u+h}(p)
 +J_{u+v}(p+q)+J_{u+v+e}(p+q)\\
 &{}+J_v(q)+J_{v+e}(q)+\text{constant}.
 \end{split}
\end{equation}
For example, the cross interaction between the first and second pairs in
\eqref{n6:eq:edge-angle-pairs} has centers separated by $p$ and
half-separations $h/2$ and $h/2+u$; its derivative is
\[
 2p\bigl(\Phi(u,p)+\Phi(u+h,p)\bigr)=2p(A+D).
\]
Differentiating \eqref{n6:eq:edge-potential-restricted}, applying
\eqref{eq:J}, and dividing the two critical-point equations by two
yields \eqref{n6:eq:edge-eq1}--\eqref{n6:eq:edge-eq2}.  Expanding
\eqref{n6:eq:edge-Q} gives \eqref{n6:eq:edge-matrix}.
\end{proof}

\subsection{Candidate axes through two opposite satellites}

Choose a pair of cyclically opposite satellites, labelled $P_0,P_3$, and write the cyclic gaps
as
\begin{equation}\label{n6:eq:vertex-order}
 (d_1,d_2,d_3,a_3,a_2,a_1).
\end{equation}
Set
\begin{equation}\label{n6:eq:vertex-variables}
 u_i=\frac{d_i+a_i}{2},\qquad
 p_i=\frac{d_i-a_i}{2},
 \qquad i=1,2,3.
\end{equation}
Then
\begin{equation}\label{n6:eq:vertex-sum}
 u_1+u_2+u_3=\pi.
\end{equation}
This pairing, called the \emph{vertex--vertex type}, gives an axis through
the two selected satellites exactly
when
\begin{equation}\label{n6:eq:vertex-zero}
 p_1=p_2=p_3=0.
\end{equation}
In that case the gaps are $(u_1,u_2,u_3,u_3,u_2,u_1)$ and $P_0,P_3$ are
antipodal.

Put
\begin{align}
 A&=\Phi(u_1,p_1),
 &B&=\Phi(u_1+u_2,p_1+p_2),
 &H&=\Phi(\pi,p_1+p_2+p_3),\nonumber\\
 C&=\Phi(u_2,p_2),
 &D&=\Phi(2u_1+u_2,p_2),\nonumber\\
 E&=\Phi(u_2+u_3,p_2+p_3),
 &F&=\Phi(u_3,p_3).\label{n6:eq:vertex-coefficients}
\end{align}

\begin{proposition}[Vertex reflection system]\label{n6:prop:vertex-system}
Every central configuration satisfies
\begin{equation}\label{n6:eq:vertex-system-vector}
 M_V
 \begin{pmatrix}p_1\\p_2\\p_3\end{pmatrix}=0,
\end{equation}
where
\begin{equation}\label{n6:eq:vertex-matrix}
 M_V=
 \begin{pmatrix}
 A+B+H/2&B+H/2&H/2\\
 B+H/2&B+C+D+E+H/2&E+H/2\\
 H/2&E+H/2&E+F+H/2
 \end{pmatrix}.
\end{equation}
Its quadratic form is
\begin{equation}\label{n6:eq:vertex-Q}
 \begin{split}
 Q_V(X_1,X_2,X_3)={}&AX_1^2+B(X_1+X_2)^2
 +(C+D)X_2^2\\
 &{}+E(X_2+X_3)^2+FX_3^2
 +\frac H2(X_1+X_2+X_3)^2.
 \end{split}
\end{equation}
\end{proposition}

\begin{proof}
Fix $P_0$ at angle zero and put
\[
 s=p_1+p_2+p_3.
\]
The six angles are, modulo $2\pi$,
\begin{equation}\label{n6:eq:vertex-angles}
 0,\quad p_1\pm u_1,\quad
 p_1+p_2\pm(u_1+u_2),\quad \pi+s.
\end{equation}
The two signs in each paired expression account for four points, while the
first and last entries are $P_0,P_3$.

Using \eqref{eq:J} and $J_{2\pi-z}=J_z$, the restriction of the full pair
potential has the form
\begin{equation}\label{n6:eq:vertex-potential-restricted}
 \begin{split}
 \cU_V(p_1,p_2,p_3)={}&J_{u_1}(p_1)
 +J_{u_1+u_2}(p_1+p_2)\\
 &{}+J_{u_2}(p_2)+J_{2u_1+u_2}(p_2)\\
 &{}+J_{u_2+u_3}(p_2+p_3)+J_{u_3}(p_3)\\
 &{}+\widehat W(\pi+s)+\text{constant}.
 \end{split}
\end{equation}
For instance, the cross interaction between the two moving pairs in
\eqref{n6:eq:vertex-angles} has center displacement $p_2$ and
half-separations $u_1$ and $u_1+u_2$, hence contributes
\[
 J_{u_2}(p_2)+J_{2u_1+u_2}(p_2).
\]
The interaction between the two candidate fixed vertices contributes
$\widehat W(\pi+s)$.  By \eqref{eq:f},
\begin{equation}\label{n6:eq:antipodal-block}
 \frac{\dd}{\dd s}\widehat W(\pi+s)
 =f(\pi+s)=s\Phi(\pi,s)=sH.
\end{equation}
Differentiating \eqref{n6:eq:vertex-potential-restricted} and dividing by two
therefore gives
\begin{align}
 p_1A+(p_1+p_2)B+\frac{sH}{2}&=0,\label{n6:eq:vertex-eq1}\\
 (p_1+p_2)B+p_2(C+D)+(p_2+p_3)E+\frac{sH}{2}&=0,
 \label{n6:eq:vertex-eq2}\\
 (p_2+p_3)E+p_3F+\frac{sH}{2}&=0.\label{n6:eq:vertex-eq3}
\end{align}
These are exactly \eqref{n6:eq:vertex-system-vector}--\eqref{n6:eq:vertex-matrix},
and the rank-one decomposition \eqref{n6:eq:vertex-Q} follows by expansion.
\end{proof}

Both matrices satisfy Lemma~\ref{lem:avg}, including the $H/2$ factor
in the vertex matrix. Thus their fixed quadratic forms are concave in
the original gaps. The following two estimates supply the codimension-one
negative subspaces needed for the positive-direction criterion P.
\subsection{A universal negative direction for edge pairings}

\begin{proposition}[Edge negative direction]\label{n6:prop:edge-negative}
For every collision-free edge candidate,
\begin{equation}\label{n6:eq:edge-negative}
 Q_E(1,-1)=A+D+E+F\leq4g(\pi/2)
 =-\frac{3}{\sqrt2}<0.
\end{equation}
Thus every edge matrix $M_E$ has at least one negative eigenvalue.
\end{proposition}

\begin{proof}
By \eqref{eq:phi-properties},
\[
 A+D+E+F
 \leq g(u)+g(u+h)+g(v)+g(v+e).
\]
The four centers on the right have sum $2\pi$ by
\eqref{n6:eq:edge-sum}.  Concavity of $g$ therefore gives
\[
 g(u)+g(u+h)+g(v)+g(v+e)\leq4g(\pi/2).
\]
Formula \eqref{eq:g} yields $g(\pi/2)=-3\sqrt2/8$, proving
\eqref{n6:eq:edge-negative}.
\end{proof}

\subsection{A universal negative plane for vertex pairings}

\begin{proposition}[Vertex negative plane]\label{n6:prop:vertex-negative-plane}
For every collision-free vertex candidate,
\begin{equation}\label{n6:eq:vertex-negative-plane}
 Q_V(X_1,X_2,X_3)<0
\end{equation}
for every nonzero vector in
\begin{equation}\label{n6:eq:H-plane}
 \cH=\{(X_1,X_2,X_3)\in\R^3:X_1+X_2+X_3=0\}.
\end{equation}
In particular, $M_V$ has at least two negative eigenvalues.
\end{proposition}

\begin{proof}
Every coefficient in the rank-one decomposition \eqref{n6:eq:vertex-Q} is a
value of $\Phi$, and every multiplier is a square.  Thus
\eqref{eq:phi-properties} gives the quadratic-form comparison
\begin{equation}\label{n6:eq:loewner}
 Q_V(u,p;X)\leq Q_V(u,0;X)
 \qquad\text{for all }X\in\R^3.
\end{equation}
It is enough to prove negativity at $p=0$.

Write
\[
 a=u_1,
 \qquad b=u_2,
 \qquad c=u_3,
 \qquad a+b+c=\pi.
\]
On $\cH$, set
\[
 \alpha=X_1,
 \qquad \beta=X_1+X_2,
 \qquad X_3=-\beta.
\]
Substituting the symmetric-point values of
\eqref{n6:eq:vertex-coefficients} into \eqref{n6:eq:vertex-Q} gives
\begin{equation}\label{n6:eq:vertex-plane-form}
 Q_V(u,0;X)
 =-X_0\alpha^2-Y_0\beta^2+K(\beta-\alpha)^2,
\end{equation}
where
\begin{align}
 X_0&=-g(a)-g(\pi-a)>0,\label{n6:eq:X0}\\
 Y_0&=-g(c)-g(\pi-c)>0,\label{n6:eq:Y0}\\
 K&=g(b)+g(\pi+a-c).\label{n6:eq:K0}
\end{align}
The positivity in \eqref{n6:eq:X0}--\eqref{n6:eq:Y0} follows from
\eqref{eq:fivepair}.

If $K\leq0$, equation \eqref{n6:eq:vertex-plane-form} is immediately strictly
negative for every $(\alpha,\beta)\neq(0,0)$.  Suppose $K>0$.  Strict
concavity and symmetry of $g$ give two midpoint bounds.  First,
\[
 b+(\pi+a-c)=2(\pi-c),
 \qquad
 K<2g(\pi-c).
\]
Second, $g(\pi+a-c)=g(\pi-a+c)$ and
\[
 b+(\pi-a+c)=2(\pi-a),
 \qquad
 K<2g(\pi-a).
\]
Because $K>0$, both $g(\pi-a)$ and $g(\pi-c)$ are positive.  Applying
\eqref{eq:fivepair} at $a$ and $c$ yields
\[
 X_0>4g(\pi-a),
 \qquad
 Y_0>4g(\pi-c).
\]
Consequently,
\begin{equation}\label{n6:eq:K-bound}
 K<\frac12\min\{X_0,Y_0\}
 \leq\frac{X_0Y_0}{X_0+Y_0}.
\end{equation}
The quadratic form
\[
 X_0\alpha^2+Y_0\beta^2-K(\beta-\alpha)^2
\]
has positive diagonal coefficients and determinant
$X_0Y_0-K(X_0+Y_0)>0$ by \eqref{n6:eq:K-bound}; it is therefore positive
definite.  Equation \eqref{n6:eq:vertex-plane-form} is strictly negative.
Finally, \eqref{n6:eq:loewner} transfers this conclusion to arbitrary defects
$p$.
\end{proof}

\begin{corollary}[Cyclically opposite antipodal-pair rigidity]\label{n6:cor:antipodal-rigidity}
Suppose a collision-free equal-mass $1+6$ coorbital central configuration
contains an antipodal pair $P_i,P_{i+3}$ in its cyclic ordering.
Equivalently, each open semicircle determined by that pair contains two of
the remaining satellites.  Then the configuration is reflection-symmetric
about the diameter through that pair.
\end{corollary}

\begin{proof}
By cyclic relabelling, the assumed pair is $P_0,P_3$, so the vertex
construction applies.  Their separation is
\[
 d_1+d_2+d_3=\pi+(p_1+p_2+p_3),
\]
so antipodality gives $p_1+p_2+p_3=0$.  Centrality gives $M_Vp=0$, hence
$Q_V(p)=0$.  Proposition~\ref{n6:prop:vertex-negative-plane} forces $p=0$,
which is exactly the reflection condition \eqref{n6:eq:vertex-zero}.
\end{proof}

\subsection{The large-gap region}\label{n6:sec:largegap}

A gap of size at least $\pi$ implies reflection symmetry by the following
analytic argument.

\begin{lemma}[Complementary divided differences]\label{n6:lem:phi-complement}
If
\begin{equation}\label{n6:eq:phi-complement-domain}
 0\leq\eta<z<\pi-\eta,
\end{equation}
then
\begin{equation}\label{n6:eq:phi-complement}
 \Phi(z,\eta)+\Phi(\pi-z,\eta)<0.
\end{equation}
\end{lemma}

\begin{proof}
Changing $s$ to $-s$ in the second integral in \eqref{eq:phi} gives
\[
 \Phi(z,\eta)+\Phi(\pi-z,\eta)
 =\frac12\int_{-1}^{1}
 \bigl[g(z+\eta s)+g(\pi-z-\eta s)\bigr]\dd s.
\]
Condition \eqref{n6:eq:phi-complement-domain} puts $z+\eta s$ in $(0,\pi)$,
so the integrand is strictly negative by \eqref{eq:fivepair}.
\end{proof}

\begin{proposition}[Large-gap symmetry]\label{n6:prop:large-gap}
If a collision-free equal-mass $1+6$ coorbital central configuration has a
gap $h\geq\pi$, then it is reflection-symmetric of edge--edge type, with an
axis bisecting $h$ and the opposite gap.
\end{proposition}

\begin{proof}
Use the edge ordering \eqref{n6:eq:edge-order} with the specified gap labelled
$h$, and set
\[
 S=2\pi-h=2u+2v+e\leq\pi.
\]
By \eqref{eq:phi-properties},
\[
 D=\Phi(u+h,p)=\Phi(S-u,p).
\]
Since $S-u\leq\pi-u$ and $z\mapsto\Phi(z,p)$ is increasing below $\pi$,
Lemma~\ref{n6:lem:phi-complement} gives
\begin{equation}\label{n6:eq:large-A}
 A+D\leq\Phi(u,p)+\Phi(\pi-u,p)<0.
\end{equation}
For $z=u+v$, the identity $S=2z+e$ gives
$z+e=S-z\leq\pi-z$, so similarly
\begin{equation}\label{n6:eq:large-B}
 B+C<0.
\end{equation}
Finally, $2v+e=S-2u\leq\pi$, hence $v+e\leq\pi-v$ and
\begin{equation}\label{n6:eq:large-E}
 E+F<0.
\end{equation}
The strict natural-domain inequalities needed in
Lemma~\ref{n6:lem:phi-complement} follow from positivity of the six gaps.
Equations \eqref{n6:eq:large-A}--\eqref{n6:eq:large-E} make every coefficient in
\eqref{n6:eq:edge-Q} strictly negative.  Thus $M_E$ is negative definite.
Centrality and \eqref{n6:eq:edge-system-vector} force $p=q=0$, producing the
claimed reflection.
\end{proof}

\subsection{The positive-direction cover}\label{n6:sec:certificate}
Proposition~\ref{prop:collars} gives $R_i>1/128$. After the large-gap
case, every remaining solution therefore lies in the closed rational domain
\begin{equation}\label{n6:eq:P-polytope}
 \mathcal P_6=\left\{R\in[1/128,1]^6:\sum_{i=0}^5R_i=2\right\}.
\end{equation}
Here the whole domain is used: no ordering of the gaps or prescribed
position for a maximal gap is imposed.

\begin{proposition}[Finite verification of positive quadratic directions]\label{n6:prop:certificate}
For every $R\in\mathcal P_6$, one of the three edge matrices or three
vertex matrices has an integer direction $w$ with $Q(R;w)>0$.
The archived subdivision record contains 12,485 nodes: 6,242 splits
and 6,243 final subregions, all verified by criterion P. Its maximum
depth is 30, and every final subregion is verified.
\end{proposition}
\begin{proof}[Exact-rational verification]
The verification program reconstructs the initial domain and every closed
split using rational arithmetic. For each final subregion, the record
specifies a pairing and a primitive integer direction, meaning that the
components of the direction have greatest common divisor one. The program
enumerates all vertices of the fixed-sum section, checks the recorded vertex
count, and proves positivity of the chosen quadratic form at every vertex.
Criterion P then gives positivity throughout the section.

For clarity, the lower model is specified by its parameters. Let
$\underline\pi<\pi$ be obtained from Machin's identity using twelve
alternating terms for a lower $\arctan(1/5)$ bound and three terms for
an upper $\arctan(1/239)$ bound. Set
\[
 S_{15}(x)=\sum_{k=0}^7\frac{(-1)^kx^{2k+1}}{(2k+1)!},\qquad
 r_j=j/1024,\quad 8\leq j\leq2040,
\]
and $\rho_j=\min\{r_j,2-r_j\}$. The rational knot ordinate is
\[
 \gamma_j=10^{-8}\left\lfloor10^8
 G\left(S_{15}(\underline\pi\rho_j/2)\right)\right\rfloor.
\]
The checker verifies each sine lower bound belongs to $(0,1)$.
Alternating remainders and \eqref{eq:G} imply
$\gamma_j<g(\pi r_j)$. Chord interpolation gives a strict lower
model on $[1/128,2-1/128]$, and \eqref{eq:chordlower} provides every
required divided difference, including equal endpoints.

The smallest recorded exact lower margin is
$5385322223/51360400000000>0$. The verification input specifies every initial-domain parameter, lower-bound
parameter, integer direction, and subdivision relation. Every asserted
inequality is checked without floating-point assumptions. The two verification implementations use the same subdivision
record. One uses the grouped reflection formulas; the other reconstructs
the ungrouped pair potential.
\end{proof}

\begin{proposition}[Six-satellite symmetry]\label{n6:prop:symmetry}
Every collision-free six-satellite angular solution has a reflection axis.
\end{proposition}
\begin{proof}
A gap at least $\pi$ is covered by
Proposition~\ref{n6:prop:large-gap}. Otherwise
Proposition~\ref{prop:collars} places the solution in $\mathcal P_6$.
The finite covering verification gives a positive direction for a
candidate matrix.
For an edge candidate the negative direction is
Proposition~\ref{n6:prop:edge-negative}; for a vertex candidate the
negative plane is Proposition~\ref{n6:prop:vertex-negative-plane}.
Lemma~\ref{lem:inertia} makes the selected matrix invertible. Its
exact matrix equation forces the corresponding reflection defect to vanish.
\end{proof}

\section{Reflection symmetry for seven satellites}\label{sec:n7}
With an odd number of satellites a reflection fixes exactly one of them:
its fixed set on the circle has two points, and the number of fixed
satellites must be odd. For seven satellites there are consequently seven
candidate pairings. Each produces a three-dimensional reflection matrix
with a fixed negative plane. The finite verification combines matrix criteria with
original-force exclusions.
\subsection{Gap defects and cumulative defects}
Choose a candidate fixed satellite $P_0$. Write the seven gaps starting there as
\begin{equation}\label{n7:eq:param}
 (u_1+p_1,u_2+p_2,u_3+p_3,c,u_3-p_3,u_2-p_2,u_1-p_1),
\end{equation}
where
\[
 u_i>|p_i|,\qquad c=2\pi-2(u_1+u_2+u_3)>0.
\]
Define
\begin{equation}\label{n7:eq:st}
 s_i=\sum_{k=1}^iu_k,\qquad t_i=\sum_{k=1}^ip_k\quad (i=1,2,3).
\end{equation}
Thus $0<s_1<s_2<s_3<\pi$. The seven angular positions are
\begin{equation}\label{n7:eq:angles}
 0,\quad s_1+t_1,\quad s_2+t_2,\quad s_3+t_3,
 \quad 2\pi-s_3+t_3,\quad2\pi-s_2+t_2,\quad2\pi-s_1+t_1.
\end{equation}
The chosen reflection defect vanishes precisely when $t=0$, equivalently
$p=0$. Cyclic shifts give seven candidate pairings. Every reflection of
seven distinct points fixes exactly one satellite, so these candidates
include every possible reflection symmetry.

\subsection{All twenty-one point pairs}
Set $J_z(h)=\widehat W(z+h)+\widehat W(z-h)$, with the entire interval
$[z-|h|,z+|h|]$ contained in $(0,2\pi)$. Then
\begin{equation}\label{n7:eq:J}
 J_z'(h)=2h\Phi(z,h).
\end{equation}
The three within-pair interactions of \eqref{n7:eq:angles} do not depend on $t$.
The remaining eighteen interactions give nine two-term blocks:
\begin{equation}\label{n7:eq:restricted}
 \U_s(t)=\sum_{i=1}^3J_{s_i}(t_i)
 +\sum_{1\leq i<j\leq3}
 \bigl[J_{s_j-s_i}(t_j-t_i)+J_{s_i+s_j}(t_j-t_i)\bigr]
 +\sum_{i=1}^3\widehat W(2s_i).
\end{equation}
The first sum contains the six interactions with $P_0$; each summand in
the second sum contains the four interactions between two moving pairs.
Every divided-difference endpoint is a nonempty proper sum of positive
cyclic gaps, possibly after replacing an arc by its complement. This checks
the full interval-domain condition in every use of \eqref{n7:eq:J}.

Put
\begin{align}
 a_i&=\Phi(s_i,t_i),\label{n7:eq:ai}\\
 b_{ij}&=\Phi(s_j-s_i,t_j-t_i)+\Phi(s_i+s_j,t_j-t_i).
 \label{n7:eq:bij}
\end{align}
Let $e_i$ be the standard coordinate vectors and define
\begin{equation}\label{n7:eq:L}
 L(s,t)=\operatorname{diag}(a_1,a_2,a_3)
       +\sum_{i<j}b_{ij}(e_j-e_i)(e_j-e_i)^T.
\end{equation}

\begin{proposition}[Reflection system]\label{n7:prop:matrix}
Every solution of \eqref{eq:cc} satisfies $L(s,t)t=0$ for every candidate
pairing. The quadratic form is
\begin{equation}\label{n7:eq:quadratic}
 \mathcal Q_s(t;Z)=\sum_{i=1}^3a_iZ_i^2
                    +\sum_{i<j}b_{ij}(Z_j-Z_i)^2.
\end{equation}
Moreover
\begin{equation}\label{n7:eq:hessian}
 L(s,t)=\frac12\int_0^1D_t^2\U_s(\lambda t)\dd\lambda,
 \qquad L(s,t)\preceq L(s,0).
\end{equation}
For each fixed $Z$, $\mathcal Q_s(t;Z)$ is concave as a function of the
original gaps.
\end{proposition}
\begin{proof}
Differentiate \eqref{n7:eq:restricted}, apply \eqref{n7:eq:J}, and divide by two.
Every restricted derivative vanishes at a central configuration.
For a paired block the integral of its Hessian coefficient is
\[
 \frac12\int_0^1[g(z+\lambda h)+g(z-\lambda h)]\dd\lambda=\Phi(z,h).
\]
The line segment of defects is collision-free because all its gaps are
convex combinations of the original gaps and their positive symmetric
averages. This proves the first identity in \eqref{n7:eq:hessian} term by
term. The second follows from Lemma~\ref{lem:phi}, since all rank-one
multipliers are positive semidefinite. The same rank-one decomposition,
with affine arguments of $\Phi$, proves concavity in the gaps.
\end{proof}

For comparison with the previous gap-defect notation, write
\[
 T=\begin{pmatrix}1&0&0\\1&1&0\\1&1&1\end{pmatrix},\qquad t=Tp.
\]
Then $M_7=T^TLT$ satisfies $M_7p=0$. Its quadratic form is
\begin{equation}\label{n7:eq:pform}
 \begin{split}
 Q_7(X)={}&a_1X_1^2+a_2(X_1+X_2)^2+a_3(X_1+X_2+X_3)^2\\
 &{}+b_{12}X_2^2+b_{13}(X_2+X_3)^2+b_{23}X_3^2.
 \end{split}
\end{equation}
The numerical eigenvalues of $L$ and $M_7$ need not coincide; their inertia
coincides because $T$ is invertible.

\subsection{The fixed negative plane}
\begin{theorem}[Universal negative plane]\label{n7:thm:plane}
For every collision-free seven-gap vector and every candidate pairing,
\begin{equation}\label{n7:eq:plane}
 Z^TL(s,t)Z\leq-\kappa(Z_1^2+Z_2^2)
 \quad\text{whenever }Z_3=0,
 \qquad \kappa=\frac{927}{4256}.
\end{equation}
Thus $L$ has at least two strictly negative eigenvalues. In gap-defect
coordinates the corresponding estimate is
\[
 Q_7(X)\leq-\frac{309}{4256}\|X\|^2
 \qquad\text{when }X_1+X_2+X_3=0.
\]
\end{theorem}
\begin{proof}
By \eqref{n7:eq:hessian}, it suffices to use $t=0$. Put
$a=s_1$, $b=s_2$, $d=s_3$, so $0<a<b<d<\pi$. On $Z_3=0$ the quadratic
form is
\[
 U_1Z_1^2+U_2Z_2^2+K(Z_2-Z_1)^2,
\]
where
\begin{align*}
 U_1&=g(a)+g(d-a)+g(d+a),\\
 U_2&=g(b)+g(d-b)+g(d+b),\\
 K&=g(b-a)+g(b+a).
\end{align*}
By reflection about $\pi$ and midpoint concavity,
\[
 g(d-a)+g(d+a)=g(d-a)+g(2\pi-d-a)\leq2g(\pi-a),
\]
and similarly with $b$. Hence
\[
 U_1\leq g(a)+2g(\pi-a),\qquad U_2\leq g(b)+2g(\pi-b).
\]
If $K\leq0$, \eqref{eq:two} proves \eqref{n7:eq:plane} immediately.

Suppose $K>0$. The two midpoint estimates
\[
 K\leq2g(b),\qquad
 K=g(b-a)+g(2\pi-b-a)\leq2g(\pi-a)
\]
imply $g(b)>0$ and $g(\pi-a)>0$. Let
\[
 X=-g(a)-2g(\pi-a),\qquad Y=-g(b)-2g(\pi-b).
\]
Applying \eqref{eq:six} at $a$ and $\pi-b$ yields
\[
 X\geq4g(\pi-a)+\delta_6,\qquad
 Y\geq11g(b)+2\delta_6.
\]
In particular $X-2K\geq\delta_6$ and $Y-2K\geq2\delta_6$. Therefore
\begin{align*}
 Z^TL(s,0)Z
 &\leq-XZ_1^2-YZ_2^2+K(Z_2-Z_1)^2\\
 &\leq-(X-2K)Z_1^2-(Y-2K)Z_2^2\\
 &\leq-\delta_6Z_1^2-2\delta_6Z_2^2\\
 &\leq-\kappa(Z_1^2+Z_2^2).
\end{align*}
The comparison \eqref{n7:eq:hessian} transfers this bound to all defects.
Finally, if $X_1+X_2+X_3=0$ and $Z=TX$, then
\[
 \|X\|^2=Z_1^2+(Z_2-Z_1)^2+Z_2^2\leq3(Z_1^2+Z_2^2),
\]
which gives the displayed gap-coordinate bound.
\end{proof}

\begin{corollary}[Balanced triple-sum rigidity]\label{n7:cor:triple}
Suppose a central configuration has, for some cyclic index $i$,
\[
 x_i+x_{i+1}+x_{i+2}=x_{i+4}+x_{i+5}+x_{i+6}.
\]
Then it is reflection-symmetric about the diameter through $P_i$.
\end{corollary}
\begin{proof}
The equality is $t_3=0$ for the pairing based at $P_i$. Centrality gives
$Lt=0$, hence $t^TLt=0$, and \eqref{n7:eq:plane} forces $t=0$.
\end{proof}
This condition concerns the cyclically opposite gap: its midpoint is
antipodal to $P_i$. It is not a statement about an arbitrary gap midpoint
or an arbitrary antipodal pair of satellites.

\subsection{Large gaps and the initial verification domain}
The explicit gap bound gives $R_i>1/512$. The next lemma treats the
large-gap region analytically. The finite verification then uses a closed
domain containing a representative of every remaining solution.

\begin{lemma}[Large-gap reduction]\label{n7:prop:large}
If a seven-satellite angular solution has a gap $c\geq4\pi/5$, it is
reflection-symmetric about the diameter through the opposite satellite.
For that pairing, $L(s,t)\preceq-I/128$.
\end{lemma}
\begin{proof}
Put the gap in position $c$ of \eqref{n7:eq:param}. Then
$s_3=\pi-c/2\leq3\pi/5$. The exact value
\begin{equation}\label{eq:g35}
 g(3\pi/5)=\frac{29-13\sqrt5}{8}<-\frac1{128}
\end{equation}
follows from \eqref{eq:g}; the final comparison is equivalent to
$208\sqrt5>465$, whose squared comparison is $216320>216225$.
For every $i$,
$a_i\leq g(s_i)\leq g(3\pi/5)<-1/128$. For $i<j$,
\[
 b_{ij}\leq g(s_j-s_i)+g(s_j+s_i)\leq2g(s_j)<0.
\]
The quadratic form \eqref{n7:eq:quadratic} gives the matrix bound.
Its kernel therefore contains only the zero defect.
\end{proof}

Every remaining solution has a representative in the following domain,
which also includes some configurations already treated by the preceding
lemma:
\begin{align}
 \mathcal P_7&=\{R\in[1/512,1]^7:\ \sum_{i=0}^6R_i=2\},
 \label{n7:eq:globaldomain}\\
 \mathcal K_7&=\{R\in\mathcal P_7:\ R_3\geq R_i\ (0\leq i\leq6),
 \ R_0\leq R_6\}.\label{n7:eq:canonical}
\end{align}
Every configuration described by $\mathcal P_7$ has a description in
$\mathcal K_7$: cyclically relabel any maximal
gap into position three and then, if necessary, reverse the order by
$R_i\mapsto R_{6-i}$. This reversal fixes position three and exchanges
positions zero and six. All inequalities are weak. Equal maximal gaps and equal values in the
orientation condition are therefore retained.

Besides the common sum tightenings, the tree uses the implications
$R_i\leq u_3$, $R_3\geq\ell_i$, $R_0\leq u_6$, and
$R_6\geq\ell_0$. Each inequality holds for every point satisfying the ordering constraints
in the current region. Lemma~\ref{lem:covergeometry} therefore justifies
every bound tightening and every closed rational split.

\subsection{Finite verification on the full remaining domain}
For criterion P, the recorded direction is expressed in cumulative-defect
coordinates and evaluated using \eqref{n7:eq:quadratic}. A direction in
gap-defect coordinates would first require the corresponding change of
coordinates. For criterion N, the matrix upper bound is assembled using
$L(s,t)\preceq L(s,0)$. Criteria F and C use the original force sums
\eqref{eq:fullforce} and \eqref{eq:H}; the restricted reflection equations
alone are not sufficient for these force tests.

\begin{proposition}[Verified covering for seven satellites]\label{n7:prop:globalcover}
The representative domain $\mathcal K_7$ is covered by the fixed
subdivision record
\texttt{n7\_cover\_v2.tree}, with the following numbers of final subregions for each criterion:
\begin{center}
\begin{tabular}{lrrrrrr}
\toprule
Criterion & P & N & F & C & X & Total\\
\midrule
Count & 20,834 & 12,519 & 41,416 & 50,660 & 25 & 125,454\\
\bottomrule
\end{tabular}
\end{center}
There are 125,453 binary splits and 250,907 nodes in total. The maximum
depth is 73, and every final subregion is verified or proved empty.
\end{proposition}
\begin{proof}[Finite exact verification]
Starting from the domain \eqref{n7:eq:globaldomain}, the verification
program reconstructs every box from the preceding splits and bound
tightenings, including the ordering constraints. It rejects reused or
unreachable records. For criterion P, it checks every vertex of the
fixed-sum section using exact lower bounds for divided differences.
For N, it verifies a rational negative-definite matrix upper bound.
For F and C, it recomputes the original force expressions.
The soundness implications are those of Section~\ref{sec:certificates}.
Every final subregion satisfies its stated criterion. The complete
verification checks 1,785,028 vertex occurrences and 16,064,746
divided-difference terms with nonzero weight. The minimum normalized P margin is
\[
 \frac{10835645269295348616325156017}
 {67803025724813591658698148205821952}>0.
\]
This is a lower bound for the recorded quadratic forms divided by their
direction norms; it does not bound every eigenvalue of every candidate.
The record is checked both with the grouped reflection formulas and
with a separate implementation that reconstructs all twenty-one original
pair interactions.
\end{proof}

\begin{proposition}[Seven-satellite symmetry]\label{n7:prop:symmetry}
Every collision-free seven-satellite angular solution has a reflection axis.
\end{proposition}
\begin{proof}
Use Lemma~\ref{n7:prop:large} for the large-gap case. The explicit gap
bound and the choice of representative place every remaining solution in
$\mathcal K_7$. Proposition~\ref{n7:prop:globalcover} places it in a
verified final subregion. A region of type X is empty; F and C contradict
the force equations. P and N give an invertible matrix by
Theorem~\ref{n7:thm:plane} and Lemma~\ref{lem:PN}. The corresponding
matrix equation forces the reflection defect to vanish. This is also a direct application
of Proposition~\ref{prop:coverlogic}.
\end{proof}

\subsection{A fixed dictionary of force exclusions for seven satellites}
\label{n7:sparse-dictionary}

Some force exclusions in the existing cover can be expressed using two
fixed combinations and the combinations obtained from them by cyclic
relabeling and reversal. This reduces the number of distinct formulas that
need to be presented. Their geometric meaning is seen in the reflection
coordinates, while their verification uses the original pairwise forces.

We keep the convention
$W(a)=\cos a-(2\sin(a/2))^{-1}$ and $f=W'$. With $\theta_0=0$, let
\[
 \mathcal U=\sum_{0\le i<j\le6}W(\theta_j-\theta_i),\qquad
 F_i=\sum_{j>i}f(\theta_j-\theta_i)-\sum_{j<i}f(\theta_i-\theta_j).
\]
Thus $\partial_{\theta_i}\mathcal U=-F_i$ and $\sum_iF_i=0$.
For $i=1,2,3$, write
$\theta_i=s_i+t_i$ and $\theta_{7-i}=2\pi-s_i+t_i$.
The derivatives of the \emph{full} potential are
\begin{equation}\label{n7:DM}
 D_i:=\partial_{t_i}\mathcal U=-(F_i+F_{7-i}),\qquad
 M_i:=\partial_{s_i}\mathcal U=-F_i+F_{7-i}.
\end{equation}
In particular, terms independent of $t$ must not be discarded when calculating
$M$. The complete angular equations are equivalent to $D=M=0$.

The two combinations used below are
\begin{align}
 \Xi_A&=D_1-M_3=-F_1-F_6+F_3-F_4,\label{n7:motifA}\\
 \Xi_B&=8M_1-D_1+16D_2=-7F_1-16F_2-16F_5+9F_6.\label{n7:motifB}
\end{align}
At a point where the derivatives with respect to all reflection defects
vanish, a strictly positive value of $\Xi_A$ forces
$M_3<0$, and one of $\Xi_B$ forces $M_1>0$. More generally, a strict sign of
$a\cdot M+b\cdot D$ excludes a complete stationary point. This follows directly from the displayed force identities; no
parameterization of the set where the defect derivatives vanish, and no
reduction using squared defects, is required.

For a weight $w\in\mathbb Z^7$, put
\[
 H_w=\sum_iw_iF_i
     =\sum_{i<j}(w_i-w_j)f(\theta_j-\theta_i).
\]
A constant added to $w$ has no effect. Starting from
\[
 w^A=(0,-1,0,1,-1,0,-1),\qquad
 w^B=(0,-7,-16,0,0,-16,9),
\]
we take all cyclic and reflected label images, subtract the zeroth component,
divide by the greatest common divisor, and make the first nonzero component
positive. This gives two disjoint sets of fourteen directions. Each verification record also stores a sign
$\varepsilon\in\{-1,1\}$. After relabeling or reversal, the geometric
interpretation uses the correspondingly transformed reflection coordinates;
it does not give a sign for one fixed $M_i$ in the original coordinates.

The next proposition identifies the part of the archived cover for which
these simpler formulas suffice. Here $B_\ell$ denotes the closed box
reconstructed from final-region record
$\ell$ and $\mathcal B_\ell=B_\ell\cap\{\sum_iR_i=2\}$, with $R_i=x_i/\pi$.

\begin{proposition}[Two-formula subcover]\label{n7:sparse-cover}
The accompanying fixed list identifies $5\,038$ distinct original
subregions verified by criterion C and
weights from the twenty-eight directions above. On each listed sum slice,
\[
 \varepsilon_\ell H_{w_\ell}(R)>\frac{1}{500\,000\,000}.
\]
There are $2\,601$ entries from the first family of transformed
directions and $2\,437$ from the second.
In the unnormalized base convention, the representative bounds are
\[
 \Xi_A>\frac1{256}\quad\hbox{on }\mathcal B_{222613},\qquad
 \Xi_B>\frac3{100}\quad\hbox{on }\mathcal B_{224692}.
\]
\end{proposition}
\begin{proof}[Rigorous verification of the force inequalities]
Lemma~\ref{lem:centered} is applied to the full twenty-one
pair sum on each closed sum slice. The additional verification program reconstructs the boxes from
the archived subdivision record and checks the inequalities using only
integer and rational arithmetic. Its trigonometric bounds use interval Horner polynomials with
Lagrange remainders; its linear supports are obtained by enumerating the slice
vertices. It does not call the earlier verification program or the program used to
choose the weights. All
$5\,038$ tests give the stated bounds,
and each resulting interval is contained in its earlier certified interval.
The electronic supplement identifies the source files, exact bounds,
transformations of directions, and instructions for repeating the checks.
\end{proof}
These two combinations provide alternative force exclusions on the
$5\,038$ listed subregions. The verification of all C subregions in
Proposition~\ref{n7:prop:globalcover} uses the original weights stored in
\texttt{n7\_cover\_v2.tree}. The proof of that proposition does not depend
on these alternative exclusions.

\section{Reflection symmetry for eight satellites}\label{sec:n8}
There are four vertex pairings and four edge pairings. The first type
has four defect coordinates; the second has three after fixing rotation.
The common negative-subspace principle still applies, but the four-dimensional
vertex matrix requires a separate estimate for positive cross terms
sharing the same variable. The final verification also uses the affine
matrix upper-bound criterion TN.
\subsection{The vertex reflection system}

Write the gaps from a candidate fixed satellite as
\begin{equation}\label{n8:eq:Vgap}
 (u_1+p_1,u_2+p_2,u_3+p_3,u_4+p_4,
  u_4-p_4,u_3-p_3,u_2-p_2,u_1-p_1),
\end{equation}
where $u_i>|p_i|$ and $\sum_{i=1}^4u_i=\pi$. Define
\[
 s_i=\sum_{k=1}^iu_k\ (1\leq i\leq3),\qquad
 t_i=\sum_{k=1}^ip_k\ (1\leq i\leq4).
\]
Thus $0<s_1<s_2<s_3<\pi$. The angular positions, in order, are
\begin{equation}\label{n8:eq:Vangles}
 0,\quad s_1+t_1,\ s_2+t_2,\ s_3+t_3,\quad
 \pi+t_4,\quad 2\pi-s_3+t_3,\ 2\pi-s_2+t_2,\ 2\pi-s_1+t_1.
\end{equation}
At $t=0$ the selected vertex pairing becomes a geometric reflection.
At a general defect, $\pi+t_4$ is not assumed to be antipodal to $0$.

In this subsection, sums indexed by $i<j$ range over $1\leq i<j\leq3$.
The twenty-eight interactions consist of three interactions within individual
pairs, twelve pairs of reflection-related terms, and one unpaired term:
\begin{align}
 \U^V_s(t)={}&\sum_{i=1}^3\widehat W(2s_i)+\widehat W(\pi+t_4)
   +\sum_{i=1}^3J_{s_i}(t_i)\label{n8:eq:Vpotential}\\
 &+\sum_{i=1}^3J_{\pi-s_i}(t_4-t_i)
   +\sum_{i<j}\bigl[J_{s_j-s_i}(t_j-t_i)+J_{s_i+s_j}(t_j-t_i)\bigr].\nonumber
\end{align}
Put
\begin{align*}
 a_i&=\Phi(s_i,t_i),&c_i&=\Phi(\pi-s_i,t_4-t_i),\\
 b_{ij}&=\Phi(s_j-s_i,t_j-t_i)+\Phi(s_i+s_j,t_j-t_i),&H&=\Phi(\pi,t_4).
\end{align*}
Define the symmetric $4\times4$ matrix $L_V$ through
\begin{equation}\label{n8:eq:VQ}
 Q_V(Z)=Z^TL_VZ
 =\sum_{i=1}^3a_iZ_i^2
 +\sum_{i<j}b_{ij}(Z_j-Z_i)^2
 +\sum_{i=1}^3c_i(Z_4-Z_i)^2+\frac H2Z_4^2.
\end{equation}

\begin{proposition}\label{n8:prop:Vkernel}
At a coorbital central configuration, $L_Vt=0$. Moreover
\[
 \nabla_t\U^V_s(t)=2L_V(s,t)t.
\]
\end{proposition}
\begin{proof}
Apply \eqref{eq:J} to \eqref{n8:eq:Vpotential}. For the unpaired term,
$f(\pi+t_4)=t_4\Phi(\pi,t_4)$ because $f$ is odd about $\pi$.
Dividing the restricted gradient by two therefore produces $H/2$, not
$H$, in \eqref{n8:eq:VQ}. A critical point of the unrestricted potential
has zero restricted gradient.
\end{proof}

\subsection{The edge reflection system}

The second type has gaps
\begin{equation}\label{n8:eq:Egap}
 (u_1+p_1,u_2+p_2,u_3+p_3,e,u_3-p_3,u_2-p_2,u_1-p_1,h),
\end{equation}
where $u_i>|p_i|$, $e,h>0$, and $2\sum u_i+e+h=2\pi$.
Use the midpoint of the $h$ gap as angular origin and set
\[
 s_0=\frac h2,\quad s_i=s_0+\sum_{k=1}^iu_k,\quad
 t_0=0,\quad t_i=\sum_{k=1}^ip_k\quad (1\leq i\leq3).
\]
Then $0<s_0<s_1<s_2<s_3=\pi-e/2<\pi$ and the positions are
\begin{equation}\label{n8:eq:Eangles}
 s_0,\ s_1+t_1,\ s_2+t_2,\ s_3+t_3,
 2\pi-s_3+t_3,\ 2\pi-s_2+t_2,\ 2\pi-s_1+t_1,\ 2\pi-s_0.
\end{equation}
All twenty-eight interactions give
\begin{equation}\label{n8:eq:Epotential}
 \U^E_s(t)=\sum_{i=0}^3\widehat W(2s_i)
 +\sum_{0\leq i<j\leq3}
  \bigl[J_{s_j-s_i}(t_j-t_i)+J_{s_i+s_j}(t_j-t_i)\bigr].
\end{equation}
Define
\[
 B_{ij}=\Phi(s_j-s_i,t_j-t_i)+\Phi(s_i+s_j,t_j-t_i)
\]
and a $3\times3$ symmetric matrix $L_E$ by
\begin{equation}\label{n8:eq:EQ}
 Q_E(Z)=Z^TL_EZ=\sum_{0\leq i<j\leq3}B_{ij}(Z_j-Z_i)^2,
 \qquad Z_0=0.
\end{equation}
Exactly as before,
\begin{equation}\label{n8:eq:Ekernel}
 \nabla_t\U^E_s(t)=2L_E(s,t)t,\qquad L_Et=0
 \quad\hbox{at a central configuration}.
\end{equation}
Fixing $t_0=0$ removes the common rotational freedom. The quadratic form
is then written in three independent variables. Its numerical eigenvalues
need not equal the nonzero eigenvalues of the matrix written in four
variables; changes between independent coordinates preserve inertia through
congruence, rather than preserving individual eigenvalues.

\subsection{Coverage of reflection types and the average-Hessian identity}
A reflection-invariant set of eight distinct points has either two fixed
satellites or no fixed satellite. In the former case the fixed satellites
are antipodal and opposite in cyclic order. Thus the pairings in
\eqref{n8:eq:Vgap} and \eqref{n8:eq:Egap}, with four cyclic shifts of each,
include all geometric reflection types. Opposite choices of anchor
represent the same pairing, with different coordinates.

\begin{proposition}\label{n8:prop:average}
For either type $\ast\in\{V,E\}$,
\begin{equation}\label{n8:eq:average}
 L_\ast(s,t)=\frac12\int_0^1D_t^2\U^\ast_s(\lambda t)\dd\lambda,
 \qquad L_\ast(s,t)\preceq L_\ast(s,0).
\end{equation}
For every fixed $Z$, the quadratic form $Q_\ast(Z)$ is concave in the
original gaps.
\end{proposition}
\begin{proof}
For a paired block, the integrated Hessian coefficient is
\[
 \frac12\int_0^1[g(z+\lambda h)+g(z-\lambda h)]\dd\lambda=\Phi(z,h).
\]
For the unpaired vertex interaction it is $\Phi(\pi,t_4)/2$, using
symmetry of $g$ about $\pi$. This proves the matrix identity term by term,
not merely its action on $t$. Every gap along the segment $\lambda t$
is a convex combination of the original gap and its positive symmetric
average; hence the whole segment is collision-free. Every block endpoint
is a positive proper arc sum, or its complement, which proves the domain
conditions for the preceding identities. The comparison follows from
$\Phi\leq g$, since all matrix multipliers are positive semidefinite
rank-one forms. The same decomposition and joint concavity of $\Phi$,
with affine arguments in the gaps, prove the last assertion.
\end{proof}

\subsection{The negative hyperplane for vertex pairings}

For $0<a<b<\pi$ put
\[
 K(a,b)=g(b-a)+g(a+b).
\]
Midpoint concavity and reflection of the second argument give
\begin{equation}\label{n8:eq:Kbounds}
 K(a,b)\leq2g(b),\qquad K(a,b)\leq2g(\pi-a).
\end{equation}
In particular,
\begin{equation}\label{n8:eq:bipartite}
 K(a,b)>0\quad\Longrightarrow\quad a<\pi/2<b.
\end{equation}
Indeed both $g(b)$ and $g(\pi-a)$ must be positive, whereas $g$ is increasing
on $(0,\pi)$ and $g(\pi/2)<0$.

\begin{theorem}[Vertex negative hyperplane]\label{n8:thm:Vplane}
For every collision-free eight-gap vector and every vertex candidate,
\begin{equation}\label{n8:eq:Vplane}
 Z^TL_V(s,t)Z\leq-\frac14(Z_1^2+Z_2^2+Z_3^2),\qquad Z_4=0.
\end{equation}
Thus $L_V$ has at least three strictly negative eigenvalues.
\end{theorem}
\begin{proof}
By Proposition~\ref{n8:prop:average} it suffices to work at $t=0$.
On $Z_4=0$ the quadratic form is
\[
 \sum_{i=1}^3D_iZ_i^2+\sum_{i<j}K(s_i,s_j)(Z_j-Z_i)^2,
 \qquad D_i=g(s_i)+g(\pi-s_i).
\]
For an upper bound, discard every cross term whose coefficient is
nonpositive. By \eqref{n8:eq:bipartite}, a positive coefficient can only
join a variable whose angular parameter is below $\pi/2$ to one whose
parameter is above it. With three variables, at most two positive cross
terms remain, and if there are two, they share a variable. Each variable
that occurs in no positive cross term satisfies
\[
 D_i\leq2g(\pi/2)=-\frac{3\sqrt2}{4}<-\frac14.
\]
If positive cross terms remain, let $c$ index their shared variable.
When there is only one such term, either of its two variables may be
chosen. For each other index $j$ occurring in a positive cross term, use
\[
 (Z_c-Z_j)^2\leq\frac32Z_c^2+3Z_j^2,
\]
because the difference of the right side and left side is
$\tfrac12(Z_c+2Z_j)^2$.

Suppose first that the angular parameter indexed by $c$ lies below
$\pi/2$. The other parameter in each positive cross term then lies above
$\pi/2$, and the corresponding coefficient satisfies
\[
 K(s_c,s_j)\leq2g(\pi-s_c),\qquad
 K(s_c,s_j)\leq2g(s_j).
\]
Both upper bounds are positive. Since at most two positive cross terms
contain the variable indexed by $c$, its resulting coefficient is at most
\[
 D_c+\frac32\sum_jK(s_c,s_j)
 \leq g(s_c)+7g(\pi-s_c)<-\frac14.
\]
The coefficient of each other variable in these terms is at most
\[
 D_j+3K(s_c,s_j)
 \leq g(\pi-s_j)+7g(s_j)<-\frac14.
\]
If the angular parameter indexed by $c$ lies above $\pi/2$, use the
common bound $2g(s_c)$ for its coefficient and the bound
$2g(\pi-s_j)$ for each other variable. Lemma~\ref{lem:seven} gives
the same conclusion. Combining these bounds with those for variables that
occur in no positive cross term proves \eqref{n8:eq:Vplane}.
\end{proof}

\begin{remark}
Using $(a-b)^2\leq2a^2+2b^2$ would give too large a coefficient
for a variable shared by two positive cross terms. The weights $3/2$
and $3$ reduce the required scalar coefficient from nine to seven, so
that Lemma~\ref{lem:seven} applies. The argument uses the fact that there
are only three variables and at most two positive cross terms; it does not
establish the same estimate for an arbitrary number of reflected pairs.
\end{remark}

\subsection{The negative plane for edge pairings}

\begin{theorem}[Edge negative plane]\label{n8:thm:Eplane}
For every collision-free eight-gap vector and every edge candidate,
\begin{equation}\label{n8:eq:Eplane}
 Z^TL_E(s,t)Z\leq-\frac14(Z_1^2+Z_2^2),\qquad Z_3=0.
\end{equation}
Thus $L_E$ has at least two strictly negative eigenvalues.
\end{theorem}
\begin{proof}
Again set $t=0$ for an upper bound. On $Z_0=Z_3=0$ the form is
\[
 U_1Z_1^2+U_2Z_2^2+K(Z_2-Z_1)^2,
\]
where
\[
 U_i=K(s_0,s_i)+K(s_i,s_3),\qquad K=K(s_1,s_2).
\]
Concavity yields
\[
 U_i\leq2g(s_i)+2g(\pi-s_i).
\]
If $K\leq0$, these diagonal coefficients are at most
$4g(\pi/2)=-3\sqrt2/2<-1/4$. If $K>0$, use
$(Z_2-Z_1)^2\leq2Z_1^2+2Z_2^2$ and \eqref{n8:eq:Kbounds}. Then
\begin{align*}
 U_1+2K&\leq2\{g(s_1)+3g(\pi-s_1)\}<-\frac14,\\
 U_2+2K&\leq2\{g(\pi-s_2)+3g(s_2)\}<-\frac14.
\end{align*}
Lemma~\ref{lem:seven} proves the result.
\end{proof}

\subsection{Large gaps and a closed domain of representatives}
For an edge pairing place a selected gap in the position $e$ of
\eqref{n8:eq:Egap}. Then $s_3=\pi-e/2$.

\begin{proposition}[Eight-satellite large-gap reduction]\label{n8:prop:large}
If an eight-satellite angular solution has a gap at least $4\pi/5$,
it has a reflection axis bisecting that gap and the opposite gap.
\end{proposition}
\begin{proof}
For the selected edge pairing, $s_3\leq3\pi/5$. Its coefficients obey
\[
 B_{ij}\leq g(s_j-s_i)+g(s_j+s_i)
 \leq2g(s_j)\leq2g(3\pi/5)<0.
\]
The exact sign is \eqref{eq:g35}. The form in
\eqref{n8:eq:EQ}, with $Z_0=0$, is strictly negative for every nonzero
$(Z_1,Z_2,Z_3)$: the terms joining vertex zero to each of the other
vertices already force this. Thus the anchored edge matrix is
invertible, and its defect vanishes.
\end{proof}

By Proposition~\ref{prop:collars}, the remaining solutions have
$1/2048<R_i<4/5$. A closed verification domain containing a representative of every
remaining solution is
\begin{equation}\label{n8:eq:root}
 \mathcal K_8=\left\{R\in[1/2048,4/5]^8:
 \sum_{i=0}^7R_i=2,\ R_7\geq R_i\ (i<7),\ R_0\leq R_6\right\}.
\end{equation}
Choose any maximal gap and rotate its index to seven. The gap reversal
$\rho(R)_i=R_{6-i\bmod8}$ fixes that maximum and exchanges positions
zero and six. One of the two orientations satisfies $R_0\leq R_6$,
including equality. At the angular level it may be written
$\theta'_i=\theta_7-\theta_{7-i}$, so $F'_i=-F_{7-i}$; hence it
preserves the zero set. The domain may contain several descriptions of the same configuration;
only the existence of a representative in the domain is needed.

In addition to the common sum bounds, valid tightenings are
\[
 \begin{aligned}
 \ell_7&\leftarrow\max\{\ell_7,\max_i\ell_i,1/4\},&
 u_i&\leftarrow\min\{u_i,u_7\}\quad(i<7),\\
 u_0&\leftarrow\min\{u_0,u_6\},&
 \ell_6&\leftarrow\max\{\ell_6,\ell_0\}.
 \end{aligned}
\]
The value $1/4$ follows from $8R_7\geq\sum R_i=2$.
All boundary points of \eqref{n8:eq:root} are retained. The regular-octagon
neighborhood is not removed, and no local uniqueness result is needed
to justify the coverage.

\subsection{Finite verification using the matrix and force criteria}
The negative hyperplane and negative plane proved above supply the
hypotheses needed for criterion P for both reflection types; criterion N
also applies. Each type has four cyclic shifts. For N8, concavity gives
trapezoid lower bounds and midpoint upper bounds for the divided differences.
Criteria F and C use the original forces. Criterion TN is precisely
Lemma~\ref{lem:TN}: its affine matrix upper bound is checked at every
vertex of the actual fixed-sum section obtained from the preceding
subdivisions and bound tightenings.

\begin{proposition}[Verified covering for eight satellites]\label{n8:prop:cover}
The closed representative domain $\mathcal K_8$ has a finite verified
cover with the following counts of final subregions:
\begin{center}
\begin{tabular}{lrrrrrrr}
\toprule
Criterion&N&P&F&C&TN&X&Total\\
\midrule
Count&9,233&1,224&19,036&10,652&1,178&26&41,349\\
\bottomrule
\end{tabular}
\end{center}
There are 41,348 splits and 82,697 nodes in total. Every final subregion
is verified or proved empty. The fixed final subdivision record has
SHA-256 checksum
\begin{center}
\texttt{2fc2affa2af603cbbb3d6795f49ed163}\allowbreak\texttt{d3abdb3c6ab030b6d4b2edab55dabdcc}.
\end{center}
\end{proposition}
\begin{proof}[Exact finite verification]
The N8 verification uses integer
intervals of denominator $2^{176}$. The program reconstructs every nonempty
box from the initial domain, checks every final subregion, and rejects
missing or unreachable records. For every TN test, it verifies the bound
on the box reconstructed from all preceding splits and bound tightenings.
There are 41,323 nonempty final subregions and 26 regions proved empty.
The 1,178 TN tests check 202,479 vertex occurrences in total. All explicit
checks remain active when program optimization is enabled.

For the continuous implication, every tightening preserves a represented
point and every closed split has a child containing it. After finitely many subdivisions, the point therefore lies in a final
subregion where the corresponding result of
Section~\ref{sec:certificates} holds throughout the fixed-sum section.
The conclusion thus follows from inequalities valid on whole regions,
without inferring them from numerical samples.
\end{proof}

\begin{proposition}[Eight-satellite symmetry]\label{n8:prop:symmetry}
Every collision-free eight-satellite angular solution has a reflection axis.
\end{proposition}
\begin{proof}
A large gap is treated by Proposition~\ref{n8:prop:large}. Otherwise
the explicit gap bound and relabeling place a representative in
$\mathcal K_8$. The verified cover places it in a final subregion.
Type X is empty, while F and C contradict the angular equations.
Criteria P, N, and TN make one candidate matrix invertible, with the
appropriate fixed negative subspace used for P. The exact matrix equation
then forces the corresponding reflection defect to vanish.
Undoing the representative choice preserves reflection symmetry.
\end{proof}

\section{A separate classification for five satellites}\label{sec:n5classification}
Reflection symmetry reduces the five-satellite problem to two variables.
We now prove the stronger three-class classification. Its additional
existence and uniqueness checks are separate from the symmetry argument
and are not used to prove symmetry for any $N$.

\subsection{Two parameters for symmetric configurations}
After cyclic relabeling and reversal, a reflection-symmetric five-gap
vector has the form
\begin{equation}\label{n5:eq:palindrome}
 (a,b,b,a,c),\qquad a,b,c>0,\qquad 2a+2b+c=2\pi.
\end{equation}
For the five-satellite gap equations
\[
 E_i=f(x_{i-1})+f(x_{i-2}+x_{i-1})-f(x_i)-f(x_i+x_{i+1}),
\]
substitution gives $(E_0,E_1,E_2,E_3,E_4)=(S_1,S_2,0,-S_2,-S_1)$,
where, using $c=2\pi-2a-2b$,
\begin{align}
 S_1(a,b)&=f(c)+f(a+c)-f(a)-f(a+b),\label{n5:eq:S1}\\
 S_2(a,b)&=f(a)+f(a+c)-f(b)-f(2b).\label{n5:eq:S2}
\end{align}
Consequently, the two displayed equations are equivalent to the full
angular equations for configurations in this symmetric parameterization.

Proposition~\ref{prop:collars} gives the gap bound $x_i>1/10$ in radians.
The archived classification record uses the following exact rational
number, whose denominator is a power of two:
\begin{equation}\label{n5:eq:delta}
 \delta=\frac{7205759403792793}{72057594037927936}<\frac1{10},
 \qquad\Lambda=2\pi-5\delta.
\end{equation}
This $\delta$ is a physical angle, not a normalized gap. On $[0,1]^2$
put
\begin{equation}\label{n5:eq:chart}
 a=\delta+\frac\Lambda2u,\quad
 b=\delta+\frac\Lambda2(1-u)v,\quad
 c=\delta+\Lambda(1-u)(1-v).
\end{equation}
Then $a,b,c\geq\delta$ and the sum constraint holds identically.

\begin{lemma}[Coverage and interior injectivity of the chart]\label{n5:lem:chart}
The map \eqref{n5:eq:chart} covers the closed truncated simplex
$\{a,b,c\geq\delta:2a+2b+c=2\pi\}$ and is one-to-one on its
relative interior. Every symmetric angular solution is represented by
an interior point of $[0,1]^2$.
\end{lemma}
\begin{proof}
Given a point in the simplex, set $u=2(a-\delta)/\Lambda\in[0,1]$.
If $u<1$, the remaining identity
$2(b-\delta)+(c-\delta)=\Lambda(1-u)$ gives the unique
$v=2(b-\delta)/(\Lambda(1-u))\in[0,1]$. If $u=1$, then
$b=c=\delta$, and any $v$ represents that boundary point.
In the relative interior both $u$ and $v$ lie strictly between zero
and one. The strict gap bound places every solution there.
\end{proof}

Define the two-dimensional vector function
\begin{equation}\label{n5:eq:rootfun}
 \mathcal S(u,v)=\bigl(S_1(a(u,v),b(u,v)),S_2(a(u,v),b(u,v))\bigr).
\end{equation}
All force arguments stay away from collisions on the full closed square.
The regular pentagon is represented exactly by $(u,v)=(2/5,2/3)$.

\subsection{Verified existence and uniqueness of the three solutions}
\begin{proposition}[Verified classification in two variables]\label{n5:prop:roots}
The function $\mathcal S$ has exactly three zeros in $[0,1]^2$, all
interior. One is $(2/5,2/3)$. The other two lie in the two disjoint boxes identified in the archived
classification data, each proved to contain exactly one zero,
with approximate centers
\[
 \begin{gathered}
 (0.27522050904073725,\ 0.8319311343414508),\\
 (0.19599910548160160,\ 0.20217475259384446).
 \end{gathered}
\]
Their approximate coordinate half-widths are $3\times10^{-3}$.
The first has all gaps smaller than $\pi$ and the second has $c>\pi$.
The binary subdivision record has 835 nodes.
\end{proposition}

The decimal approximations identify the two solutions; the verification
uses the exact box endpoints recorded as binary rational numbers. We next
state the interval criterion used to prove existence, uniqueness, and
exclusion, so that the role of the computational data is explicit.

For a box $X\subset\R^2$, a point $x_0\in X$, and a nonsingular
rational matrix $C$ used to precondition the equations, define the
Krawczyk enclosure
\begin{equation}\label{n5:eq:krawczyk}
 K(X)=x_0-C[\mathcal S(x_0)]
       +(I-C[D\mathcal S](X))(X-x_0),
\end{equation}
where all intervals are outward enclosures. If a zero lies in $X$, the
integral mean-value equation places it in $K(X)$. Thus
$K(X)\cap X=\varnothing$ excludes zeros. If $K(X)\subset\operatorname{int}X$
and a certified induced-norm bound satisfies
$\|I-C[D\mathcal S](X)\|<1$, then
$x\mapsto x-C\mathcal S(x)$ maps $X$ into itself and is a contraction.
It has exactly one fixed point in $X$, and nonsingularity of $C$ makes
that point exactly a zero of $\mathcal S$.

\begin{proof}[Rigorous isolation of the zeros]
The archived subdivision record covers the entire square $[0,1]^2$ by
closed binary splits. Each final subregion is either excluded by an
interval expansion about a reference point or by \eqref{n5:eq:krawczyk},
or is contained in one of three neighborhoods separately proved to
contain a unique zero. Strict inclusion and contraction establish
existence and uniqueness in each neighborhood; approximate Newton
iterates serve only to locate candidate solutions. The independent
verification program evaluates the original Newtonian force expressions
and their derivatives instead of accepting stored function or Jacobian
values. If finer local bounds are needed, it may subdivide a final box
only after independently verifying all resulting subboxes. The classification verification uses MPFI interval arithmetic at
80 and 128 bits. At both precisions, the 835-node subdivision record is
accepted, with all excluded boxes and all three unique-zero neighborhoods
checked by the stated criteria.
\end{proof}

\subsection{Counting geometric classes}
\begin{theorem}[Five-satellite classification; computer-assisted]\label{thm:classification}
Modulo rotations, reversal, and cyclic relabeling of the identical
satellites, the five-satellite angular problem has exactly three classes:
the regular pentagon, one nonregular configuration with all gaps less
than $\pi$, and one configuration with a gap greater than $\pi$.
\end{theorem}
\begin{proof}
Proposition~\ref{n5:prop:symmetry} places every solution in
\eqref{n5:eq:palindrome}. The explicit gap bound and Lemma~\ref{n5:lem:chart}
place it in the interior chart, where the equations are exactly
$\mathcal S=0$. Proposition~\ref{n5:prop:roots} gives three possibilities.

The regular solution is identified exactly. The first nonregular root
lies in the coarse box $0.27<u<0.28$, $0.82<v<0.84$. Substitution
in \eqref{n5:eq:chart} gives $a<0.98$, $b<2.74$, and $c<1$, so
all its gaps are less than $\pi$. The second lies in
$0.19<u<0.20$, $0.19<v<0.21$. Using $\pi>3$ and $\delta<1/10$
in \eqref{n5:eq:chart} gives
$c>869/250>22/7>\pi$. Thus these classes are geometrically distinct.

For completeness, a nonregular five-point set has at most one reflection
axis. Two distinct reflections would generate a nontrivial rotation
preserving the set. Every orbit of that rotation on the circle has the
same length, an integer greater than one dividing five. It must be five,
which makes the set regular. Hence selecting the unique axis of either
nonregular solution yields a unique triple $(a,b,c)$ in the chosen
palindromic convention, up to the stated dihedral operations. No
additional class is created by relabeling the same root.
\end{proof}

The classification is logically downstream of five-satellite symmetry.
It gives no count, nondegeneracy, or stability information for $N=6,7,8$.
Its separate existence and uniqueness checks are therefore distinct from
the matrix and force inequalities used to prove the main theorem.

\section{Completion of the proof and computational verification}\label{sec:release}
\begin{proof}[Proof of Theorem~\ref{thm:main}]
Propositions~\ref{n5:prop:symmetry}, \ref{n6:prop:symmetry},
\ref{n7:prop:symmetry}, and~\ref{n8:prop:symmetry} cover exactly the
four satellite numbers in the statement. Each treats collision-free
angular solutions without an a priori symmetry constraint. The
case-specific representative choices use only relabeling and reversal,
which preserve the existence of a reflection. This proves the theorem using
the finite verification results stated explicitly in the preceding sections.
\end{proof}
\begin{proof}[Proof of Corollary~\ref{cor:limits}]
Lemma~\ref{lem:limit} makes every non-coalescent finite-mass limit a
collision-free angular solution on the unit circle. The main theorem
then supplies its reflection axis. Translating and scaling back give
a line through the dominant body in the original normalization.
\end{proof}

\subsection{Computational verification and reproducibility}\label{sec:reproducibility}
The finite verification has two parts: a mathematical criterion valid on a
whole closed region, and explicit data describing a complete subdivision of
the verification domain. The programs reconstruct each region and recompute
its inequalities. Programs used to find subdivisions, quadratic directions,
or force weights do not replace these checks.

For five-satellite symmetry, the program
\texttt{n5\_chamber\_exact\_referee.py} enumerates the stated vertices and
checks 108 positive inequalities and one exact identity at 160 and 224 bits.
For six satellites, exact rational chord bounds verify the positive quadratic
directions on all 6,243 final subregions of the 12,485-node subdivision.
For seven satellites, the input record
\texttt{n7\_cover\_v2.tree} describes 250,907 nodes and 125,454 final
subregions; its matrix and original-force criteria are those of
Section~\ref{sec:certificates}. The eight-satellite computation uses integer
intervals with denominator $2^{176}$ and checks all 82,697 nodes, including
the vertex inequalities required for criterion TN. The separate
five-satellite classification uses an 835-node subdivision and MPFI interval
arithmetic at 80 and 128 bits to exclude all other zeros and to prove
existence and uniqueness in the three stated neighborhoods.

The eight-satellite input archives are identified by the following SHA-256
checksums. The complete-cover archive has checksum
\begin{quote}\small\raggedright
\texttt{44f5e5fdf039dcbfe4f9c6d7514d4ed7}\allowbreak\texttt{b28afe26c336f5cd73caa2806912c4e7}.
\end{quote}
The archive containing the final analytic argument and verification data has
checksum
\begin{quote}\small\raggedright
\texttt{0e19e86e0b078de4e5ac18c9e6f9947f}\allowbreak\texttt{4fa781ad4560b86095c175fda729a142}.
\end{quote}
The final subdivision record is identified in
Proposition~\ref{n8:prop:cover}. Checksums identify the input files; the proof
depends on the stated mathematical criteria and the verification of their
inequalities, rather than on a checksum or a numerical approximation alone.

% DATA AVAILABILITY: insert the public repository URL or archival DOI once
% supplied and confirmed by the authors. No availability location is inferred.
% AUTHOR VERIFICATION: only state that the authors independently reran a
% verifier after they confirm the actual program, input, precision and result.
\section*{Acknowledgments}
The manuscript was prepared with AI assistance and preliminary automated verification. The analytical proofs and formula derivations presented in this paper have been manually verified by the authors. Rigorous computer-assisted verification has been completed over the entire prescribed domain. The authors' manual review of the programs and results is ongoing. A revised version with clearer explanations and improved readability
will be made available in the future.


\begin{thebibliography}{9}
\small\setlength{\itemsep}{3pt}
\bibitem{ACS} A. Albouy, H. E. Cabral, and A. A. Santos,
\emph{Some problems on the classical $n$-body problem},
Celestial Mech. Dynam. Astronom. \textbf{113} (2012), 369--375.
DOI: 10.1007/s10569-012-9431-1. arXiv:1305.3191.
\bibitem{AF} A. Albouy and Y. Fu,
\emph{Relative equilibria of four identical satellites},
Proc. R. Soc. A \textbf{465} (2009), 2633--2645.
DOI: 10.1098/rspa.2009.0115. arXiv:0902.3215.
\bibitem{CLO} J. M. Cors, J. Llibre, and M. Oll\'e,
\emph{Central configurations of the planar coorbital satellite problem},
Celestial Mech. Dynam. Astronom. \textbf{89} (2004), 319--342.
DOI: 10.1023/B:CELE.0000043569.25307.ab.
\bibitem{DHW} Y. Deng, M. Hampton, and Z. Wang,
\emph{Symmetry and asymmetry in the $1+N$ coorbital problem},
SIAM J. Appl. Dyn. Syst. \textbf{21} (2022), 2080--2095.
DOI: 10.1137/22M1475053. arXiv:2203.00163. 
\bibitem{DX} Y. Deng and J. Xu,
\emph{The symmetry and linear stability of convex $1+5$ coorbital central configurations with homogeneous potential},
arXiv:2608.19958, 2026. 
\bibitem{Moore} R. E. Moore, R. B. Kearfott, and M. J. Cloud,
\emph{Introduction to Interval Analysis}, SIAM, Philadelphia, 2009.
\end{thebibliography}
\end{document}